\documentclass[fleqn,10pt]{wlscirep}
\usepackage[utf8]{inputenc}
\usepackage[T1]{fontenc}
\usepackage{amsmath}
\usepackage{amssymb}
\usepackage{amsthm}
\usepackage{array}
\usepackage{color}
\usepackage{graphicx}
\usepackage[utf8]{inputenc}
\usepackage{multirow}
\usepackage{url}
\usepackage{blindtext}
\usepackage[square,numbers]{natbib}
\usepackage{ulem}

\newtheorem{rem}{\bf Remark}[section]
\newtheorem{thm}{Theorem}[section]
\newtheorem{cor}[thm]{Corollary}

\newtheorem{property}[thm]{Property}
\newtheorem{prop}[thm]{Proposition}

\newtheorem{definition}{Definition}

\newcommand{\adim}{\operatorname{adim}}

\title{On the $(k,\ell)$-anonymity of networks via their $k$-metric antidimension}

\author[1]{Elena Fern\'andez}
\author[1]{Dorota Kuziak}
\author[1,*]{Manuel Munoz-Marquez}
\author[2]{Ismael G. Yero}
\affil[1]{Departamento de Estad\'istica e Investigaci\'on Operativa, Universidad de C\'adiz, Spain}
\affil[2]{Departamento de Matem\'aticas, Universidad de C\'adiz, Spain}

\affil[*]{manuel.munoz@uca.es}


\begin{abstract}
This work focuses on the $(k,\ell)$-anonymity of some networks as a measure of their privacy against active attacks. 
Two different types of networks are considered. The first one consists of graphs with a predetermined structure, namely cylinders, toruses, and $2$-dimensional Hamming graphs, whereas the second one is formed by randomly generated graphs. In order to evaluate the $(k,\ell)$-anonymity of the considered graphs, we have computed their $k$-metric antidimension. To this end, we have taken a combinatorial approach for the graphs with a predetermined structure, whereas for randomly generated graphs we have developed an integer programming formulation and computationally tested its implementation. The results of the combinatorial approach, as well as those from the implementations indicate that, according to the $(k,\ell)$-anonymity measure, only the $2$-dimensional Hamming graphs and some general random dense graphs are achieving some higher privacy properties.\\

{\bf AMS Subj. Class. (2020)}: 05C12; 05C76; 90C27; 90C10; 05C80
\end{abstract}
\begin{document}

\flushbottom
\maketitle

\thispagestyle{empty}

\section{Introduction}\label{sect:k_anti}

Social network analysis is the process of investigating social structures while making use of networks and graph theory methods. 
Such analysis is widely developed in our modern society. 
This is motivated by several factors including for instance the increasing need of advances in basic computing and information technologies. These advances aim at improving systems and services frequently integrated within a variety of important business and societal functions, like e-commerce, health care, education, manufacturing, and personal interactions, among others.

The usefulness of social network analysis is doubtless, since many different social services can benefit from these investigations. These benefits are, however, not cost-free, as the privacy of the users of a network would be compromised if some involved entity could deliver sensitive data such as e-mails, instant messages, or relationships. 
A basic solution to the above issue prevents such a risk by applying some anonymization process to the released social network. For example, potentially identifying attributes can be removed. Nonetheless, this naive approach is usually not enough to guarantee the privacy of  users personal information in a network. In fact, there is always some potential probability of disclosing any user. Therefore, it would be highly desirable that every public social network would incorporate a measure of this \emph{disclosing} probability.

Some primitive approaches to measuring the above probability were first proposed in \cite{Trujillo-Rasua-2016} and later slightly improved in \cite{Mauw-2019}. 
The main idea of these approaches was to introduce a measure against active privacy attacks for social networks called $(k,\ell)$-anonymity. 
Specifically, in \cite{Trujillo-Rasua-2016} it was stated that a social graph achieving $(k,\ell)$-anonymity satisfies that the probability of disclosing any given user of that network, in the presence of at most $\ell$ attacker nodes in the network, is $1/k$. 
The $(k,\ell)$-anonymity is theoretically supported by a graph theory parameter called $k$-metric antidimension, which was introduced in \cite{Trujillo-Rasua-2016} as well. 

In the definition of the $(k, \ell)$-\emph{anonymity}, the integer $k$ is used as a privacy threshold, whereas the value $\ell$ stands for an upper bound on the expected number of attacker vertices in a given network. 
Since an attacker entity cannot easily control many vertices of the network, it is usually accepted that the number of attacker nodes in a network is likely to be significantly smaller than the total number of vertices, so it is assumed that $\ell$ is a small integer number. 

Given the tight relationship between the above two concepts, a key point for stating the $(k, \ell)$-anonymity of a given graph $G$ is finding its $k$-metric antidimension. This research topic has received the attention of several researchers. 
The $k$-metric antidimension has been studied as a graph parameter separately from the $(k,\ell)$-anonymity in several recent works, like for instance in~\cite{Cangalovic-2018,Kratica-2019,Mauw-2019,Mauw-2018,Tang-2021,Trujillo-Rasua-2016-a}, and also in the more general survey \cite{Kuziak-2021}, which contains a compilation of the main contributions on this topic. 
Further studies focusing on the combinatorial and computational properties of the $k$-metric antidimension of a graph have contributed to a better application and understanding of the $(k,\ell)$-anonymity.
For instance, the $k$-metric antidimension of some classical families of graphs like complete bipartite, cycles, and some others, was considered in the seminal paper \cite{Trujillo-Rasua-2016}; that of some generalized Petersen graphs was studied in \cite{Kratica-2019}; of some wheel-related social graphs in \cite{Tang-2021}; and of some trees and unicyclic graphs (for $k=1$) in \cite{Trujillo-Rasua-2016-a}. 
For a fairly complete compilation of results on this topic, we address the reader to the recent survey \cite{Kuziak-2021}, which deals with several other related topics as well. 

Existing work indicates that, unfortunately, for a large number of networks the value of the probability $1/k$ is much higher than what could be considered reasonably good. Some examples were given, for instance,  in \cite{DasGupta-2019} where the value of the $(k, \ell)$-anonymity achieved by a number of real social networks was presented.
In particular, most social networks only satisfy $(1,1)$-anonymity. 
This is clearly negative as it means that in most social networks (graphs) there is a single vertex with the property of potentially disclosing at least one vertex of the graph. In other words, most social networks (graphs) are quite vulnerable.

From the discussion above, we may deduce that it is clearly desirable to identify graphs that satisfy $(k, \ell)$-anonymity for some $k$ as large as possible, and any given value of $\ell$. 
Therefore, despite the existing results, there is still a need to study (compute or bound) the $k$-metric antidimension for graph classes for which this parameter remains unknown, so their $(k,\ell)$-anonymity can be established. Furthermore, the use of new methods for making such computations deserves attention. The present work pursues these two goals: to provide more insight on the knowledge of the $k$-metric antidimension, by considering graphs with a predetermined structure as well as   randomly generated ones, and  to develop a new perspective for addressing this topic. 

In particular, the objective of this paper is twofold. On the one hand, it presents an evaluation of privacy properties derived from the $(k,\ell)$-anonymity measure, for several types of networks.  
As a previous step for this evaluation, we have obtained new theoretical results on the $k$-metric antidimension. 
For several types of networks with a specific structure, these results have been obtained using combinatorial techniques, whereas for the case of randomly generated networks we have resorted to optimization. Specifically, we propose an integer programming mathematical optimization formulation to obtain the $(k, \ell)$-anonymity for a general network, which can be solved with an off-the-self solver. 
To the best of our knowledge, this methodology has not yet been applied in this context.

The remainder of this paper is structured as follows. 
Section \ref{sect:defin} introduces the notation that we will use and gives the basic definitions and results. 
Then, in Section \ref{sec:Cyl-To-Ham}, we give the exact values of the $k$-metric antidimension of toruses, cylinders and $2$-dimensional Hamming graphs for the corresponding feasible values of $k$. To facilitate the reading of the paper the proofs of all the results of this section are given in the Appendix.
In Section \ref{sect:k_ILP}  we introduce an integer programming formulation for the $k$-metric antidimension problem. 
This formulation has been implemented and computationally tested on some regular and almost regular random networks.
The random graphs generated have three different structures. We first consider random trees, next some random general sparse graphs and, finally, some random dense graphs.
The computational experiments are described in Section \ref{sect:compu}, where we also present and analyze the obtained numerical results. Based on these results, we then state the $(k, \ell)$-anonymity satisfied by such random networks. The paper ends in Section \ref{sect:conclu} where we derive some conclusions and avenues for future research.

\section{Definitions and basic concepts}\label{sect:defin}

Throughout this paper we consider an undirected non-weighted and connected graph  $G=(V, E)$ with $|V|=n$ without loops or multiple edges. 
We further assume that the users are represented in a graph by its vertices. While the edges represent some kind of relationship between the users.
The length of any given path in $G$ is given by its number of arcs and, for any pair of vertices $u,v\in V$, the distance between $u$ and $v$ is the length of a shortest $u,v$-path, which is denoted by $d(u,v)$ (or $d_{uv}$ for short).
For a given vertex $v\in V(G)$, the \emph{eccentricity} of $v$ is $e(v)=\max\{d(v,x)\,:\,x\in V(G)\}$. 
Moreover, for the vertex $v$, any vertex $u$ such that $d(v,u)=e(v)$ is called an \emph{eccentric vertex} of $v$. Note that these are the furthest possible nodes from $v$. 
Now, the following concepts are the key point of our whole work.

\begin{definition}[\cite{Trujillo-Rasua-2016}]
\label{def:ars_adm}
Let $k\in\mathbb{N}$ and let $G=(V,E)$ be a connected graph.
\begin{itemize}
\item A set $S\subset V$ is a $k$-\emph{antiresolving set} $($$k$-\emph{ARS} for short$)$ for $G$, if $k$ is the largest integer such that for all $u\notin S$ there exists a set $S_u\subseteq V\setminus (S\cup \{u\})$ with $|S_u|\geq k-1$ and where $d_{uv}=d_{xv}$ for every $v\in S$ and every $x\in S_u$.  
\item The $k$-\emph{metric antidimension} of $G$, denoted $\adim_k(G)$, is the cardinality of a smallest $k$-ARS for $G$.
\item A $k$-ARS of cardinality $\adim_k(G)$ is called a $k$-\emph{antiresolving basis} $($$k$-\emph{ARB} for short$)$.
\end{itemize}
\end{definition}

With the above concepts, Trujillo-Rasua \& Yero \cite{Trujillo-Rasua-2016} introduced the following measure corresponding to a probability index that quantifies how secure a network is with respect to active attacks to its privacy.

\begin{definition}[\cite{Trujillo-Rasua-2016}]
\label{def:k-ell-anonymity}
A graph $G$ meets $(k, \ell)$-\emph{anonymity} with respect to active attacks, if $k$ is the smallest positive integer such that the $k$-metric antidimension of $G$ is not larger than $\ell$.
\end{definition}

The concepts from Definition~\ref{def:ars_adm} can be also seen from a different perspective.
For a given vertex set $S\subset V$, we define the following equivalence relation $\mathcal{R}_S$.
Two vertices $x,y\in V$ are related by $\mathcal{R}_S$ if for every vertex $z\in S$ it follows that $d_{xz}=d_{yz}$.
From now on, for a given set $S\subset V$, we consider $\mathcal{Z}_S=\{Z^1,\dots,Z^r\}$, for some $r\ge 1$, as the set of equivalence classes defined by $\mathcal{R}_S$.
By using the terminology above, it is readily seen that any vertex set $S\subseteq V$ is a $k$-ARS for a graph $G$ with $k=\min\{|Z^i|\,:\,Z^i\in \mathcal{Z}_S\}$.
This approach to the definition of $k$-ARS turns out to be in general more useful while dealing with this topic.
Since $\bar{r}=\frac{n}{k}-1$ is an upper bound on the number of non-empty sets $Z^r$ defined above, in the following we assume that $R=\{1, \dots, \overline{r}\}$.\\

Observe that, according to Definition \ref{def:k-ell-anonymity}, the problem of determining $\adim_k(G)$ for a given graph $G$ can be stated as the following optimization problem: 
\begin{equation}
    \tag{$k$-MAD}
   \adim_k(G)\,=\, \min \{|S|: S \mbox{ is a $k$-ARS  for } G\}\label{$k$-MAD Problem}
\end{equation}

\medskip
We first recall that solving $k$-MAD is in general NP-hard, as independently proved in \cite{Chatterjee-2019} and \cite{Zhang-2017}. 
Moreover, it is clear that a given graph $G$ may not contain a $k$-ARS for every value $k$. 
In connection with this, it is said that a graph $G$ is $\kappa$-metric antidimensional if $k=\kappa$ is the largest integer for which $G$ contains a $k$-ARS. 
In contrast to \eqref{$k$-MAD Problem}, which is NP-hard, it was proved in \cite{Chatterjee-2019} that the value $\kappa$ can be found polynomially in the order $n$ of $G$. However, knowing that a graph $G$ is $\kappa$-metric antidimensional does not imply that there exists a $k$-ARS for another $k\in \{2,\dots,\kappa-1\}$ (note that we can always find a $1$-ARS). 
One reason for this relies on the fact that there is no monotonicity (with respect to $k$) for $\adim_k(G)$. 
Nevertheless, there are families of graphs for which the monotonicity of $\adim_k(G)$ relative to $k$ has been confirmed. 
For instance, in \cite{DasGupta-2019} it was proved that a $\kappa$-metric antidimensional tree $T$ contains a $k$-ARS for every $k\in \{1,\dots,\kappa\}$. 
Still, even if a tree $T$ contains a $k$-ARS for every $k\in \{1,\dots,\kappa\}$, we must also recall that a general procedure for computing an optimal $k$-ARSs for an arbitrary value of $k$ in the interval $\{1,\dots,\kappa\}$ is not yet known. 
That is, for general trees it is not known how difficult it is to find the $k$-metric antidimension. 
This supports the interest of considering the case of randomly generated trees and computing their $k$-metric antidimension.

Regarding this situation, the existence of $k$-ARSs can be dealt with the following combinatorial problem:
\begin{equation}
   \mbox{For a given integer $k$ and a graph $G$: Does $G$ contains a $k$-ARS?}
   \tag{$k$-ARS}
   \label{$k$-ARS Problem}
\end{equation}

In the remainder of this paper, when a graph $G$ has no $k$-ARS for some $k\in \{1,\dots,\kappa\}$, we will say that $\adim_k(G)=+\infty$.

The optimization problem \eqref{$k$-MAD Problem} and the combinatorial one \eqref{$k$-ARS Problem} above shall be considered during our whole exposition. Specifically, for any studied graph $G$, we shall compute the largest value $\kappa$ for which $G$ is $\kappa$-metric antidimensional. This will provide us the interval of integers in which the $k$-metric antidimension might be computed. We shall then consider such interval $\{1,\dots,\kappa\}$ for $G$, and study the existence of $k$-ARSs for every $k\in\{1,\dots,\kappa\}$. Hence, whenever possible, we will compute the value of $\adim_k(G)$. With the obtained results, the $(k, \ell)$-anonymity satisfied for the corresponding graphs shall be stated. 

\section{Networks with a predetermined structure}\label{sec:Cyl-To-Ham}

This section considers some graphs that are obtained as the Cartesian product of some other graphs. Such graph products are constructions that are widely used in several investigations covering theoretical studies as well as applied ones. They have a symmetric structure, which intuitively makes them good candidates to satisfy $(k, \ell)$-anonymity for large values of $k$. The Cartesian product of two graphs $G$ and $H$, is the graph $G\Box H$ with vertex set $V(G\Box H)=V(G)\times V(H)$. Two vertices $(g,h),(g',h')$ are adjacent in $G\Box H$ if either $g=g'$ and $hh'\in E(H)$, or $gg'\in E(G)$ and $h=h'$. For a comprehensive compendium on product graphs, their structure, applications, recognition, etc. we suggest the book \cite{Hammack-2011}. 

We first recall well-known results for grid graphs, whose $k$-antidimension is known, and then we focus on cylinders, toruses and $2$-dimensional Hamming graphs, for which the $(k, \ell)$-anonymity has not yet been studied. The \emph{grid graph} $P_r\Box P_s$ is the Cartesian product of two paths $P_r=u_1u_2\cdots u_r$ and $P_s=v_1v_2\cdots v_s$. 

The $k$-metric antidimension of the grid graph $P_r\Box P_s$ (for $r,t\ge 2$) was studied in \cite{Cangalovic-2018}, and it is as follows. First, $P_r\Box P_s$ is $4$-metric antidimensional when $r,s$ are both odd; and otherwise it is $2$-metric antidimensional. When $r,s$ are both odd it is not possible to find $3$-ARSs. In consequence, the following formulas are known (see \cite{Cangalovic-2018}).
\begin{itemize}
  \item $\adim_1(P_r\Box P_s)=1$ 
  \item $\adim_2(P_r\Box P_s)=\left\{\begin{array}{ll}
                            2, & \mbox{if $r,s$ are both even,} \\
                            1, & \mbox{otherwise.} \\
                          \end{array}
\right.$ 
  \item $\adim_4(P_r\Box P_s)=1$ (if $r,s$ are both odd). 
\end{itemize}

From the results above, one can deduce that the grids $P_r\Box P_s$ are only satisfying $(1,1)$-anonymity, since (in the presence of already one attacker node) the smallest value for which $\adim_k(P_r\Box P_s)\le 1$ is $k=1$.\\

We next consider the cylinder $P_r\Box C_s$, the torus $C_r\Box C_s$ and the so-called $2$-dimensional Hamming graph $K_r\Box K_r$, where, following the usual notation, $C_r=u_0u_1\cdots u_{r-1}u_0$ and $K_r$, respectively represent a cycle and a complete graph of order $r$. We first find the values $\kappa$ for which such graphs are $\kappa$-metric antidimensional and then establish their $(k,l)$-anonymity. 
In order to facilitate the flow of our exposition, we present in this section the obtained results and include all the proofs in the Appendix.

\begin{prop}
\label{prop:k_antidim_toru_cyl_Ham}
Let $r,s\ge 2$ be two integers.
\begin{itemize}
  \item[{\em (i)}] If $r\ge 2$ and $s\ge 3$, then $P_r\Box C_s$ is $4$-metric antidimensional if $r,s$ are odd, it is $3$-metric antidimensional if $s$ is even, and otherwise, it is $2$-metric antidimensional.
  \item[{\em (ii)}] If $r,s\ge 3$, then $C_r\Box C_s$ is $4$-metric antidimensional if $r,s$ have the same parity, and otherwise, it is $2$-metric antidimensional.
  \item[{\em (iii)}] If $r\ge 4$, then $K_r\Box K_r$ is $(2r-2)$-metric antidimensional.
\end{itemize}
\end{prop}

In concordance with Proposition \ref{prop:k_antidim_toru_cyl_Ham}, we next give the exact values of the $k$-metric antidimension of cylinders, toruses and $2$-dimensional Hamming graphs for those suitable values of $k$, namely those values of $k$ not larger than $\kappa$ (for each corresponding graph). We begin with the cylinder $P_r\Box C_s$.

\begin{thm}
\label{th:k_antidim_cyl}
For every integers $r\ge 2$ and $s\ge 3$,
$$\adim_k(P_r\Box C_s)=\left\{\begin{array}{ll}
                                1, & \mbox{if $k=4$ and $r,s$ are odd}, \\
                                2, & \mbox{if $k=3$ and $s$ is even}, \\
                                +\infty, & \mbox{if $k=3$ and $r,s$ are odd}, \\
                                1, & \mbox{if $k=2$ and $r,s$ are not both even},\\
                                4, & \mbox{if $k=2$ and $r,s$ are even},\\
                                2, & \mbox{if $k=1$ and $s$ is odd}, \\
                                1, & \mbox{if $k=1$ and $s$ is even}.
                              \end{array}\right.
  $$
\end{thm}

The next result presents the $k$-metric antidimension of the torus $C_r\Box C_s$, in concordance with  Proposition \ref{prop:k_antidim_toru_cyl_Ham} (ii).

\begin{thm}
\label{th:k_antidim_toru}
For every two integers $r,s\ge 3$,
$$\adim_k(C_r\Box C_s)=\left\{\begin{array}{ll}
                                1, & \mbox{if $k=4$ and $r,s$ are odd}, \\
                                2, & \mbox{if $k=4$ and $r,s$ are even}, \\
                                4, & \mbox{if $k=3$ and $r,s$ are even}, \\
                                +\infty, & \mbox{if $k=3$ and $r,s$ are odd}, \\
                                1, & \mbox{if $k=2$ and $r,s$ have distinct parity}, \\
                                4, & \mbox{if $k=2$ and $r,s$ are even},\\
                                \min\{r,s\}, & \mbox{if $k=2$ and $r,s$ are odd},\\
                                1, & \mbox{if $k=1$ and $r,s$ are even}, \\
                                2, & \mbox{if $k=1$ and $r,s$ are not both even}.
                              \end{array}\right.
  $$
\end{thm}

In the next result we compute the $k$-metric antidimension of the $2$-dimensional Hamming graph for the suitable values of $k$, according to Proposition \ref{prop:k_antidim_toru_cyl_Ham} (iii).

\begin{thm}
\label{th:k_antidim_Hamming}
For every $r\ge 4$,
$$\adim_k(K_r\Box K_r)=\left\{\begin{array}{ll}
                                3, & \mbox{if $k=1$}, \\
                                2, & \mbox{if $k=2$}, \\
                                r-k, & \mbox{if $3\le k\le r-2$}, \\
                                r, & \mbox{if $k=r-1$}, \\
                                +\infty, & \mbox{if $r\le k\le 2r-3$}, \\
                                1, & \mbox{if $k=2r-2$}.
                              \end{array}\right.
  $$
\end{thm}

\subsection{The $(k,\ell)$-anonymity of cylinders, toruses and $2$-dimensional Hamming graphs}

Theorems \ref{th:k_antidim_cyl}, \ref{th:k_antidim_toru} and \ref{th:k_antidim_Hamming} lead to the following measures for the corresponding graphs studied in each case, under the assumption of the existence of one attacker vertex ($\ell=1$).

\begin{table}[ht]
\newcolumntype{C}{>{\centering\arraybackslash}p{6.5em}}
\centering\small
 \begin{tabular}{|c|c|c|c|c|}
    \hline
    \multicolumn{2}{|c|}{Cylinder graph} & \multicolumn{2}{c|}{Torus graph} & Hamming graph\\
    \multicolumn{2}{|c|}{$P_r\Box P_s$} & \multicolumn{2}{|c|}{$C_r\Box P_s$} & $K_r\Box K_r$ \\ \hline
    $s\ge 4$ even  & $s\ge 3$ odd & $r,s\ge 4$ even & otherwise & $r\ge 4$ \\
    \hline
    $(1,1)$-anonymity  & $(2,1)$-anonymity & $(1,1)$-anonymity & $(2,1)$-anonymity & $(2r-2,1)$-anonymity \\
    \hline
\end{tabular}
  \caption{Anonymity achieved by toruses, cylinders and $2$-dimensional Hamming graphs.}\label{tab:anon-CTH}
\end{table}

In addition to the results in Table \ref{tab:anon-CTH}, we can also note that, for instance, if $r,s$ are both odd, then the torus graph $C_r\Box C_s$ achieves $(4,1)$-anonymity in the presence of only one attacker vertex. Therefore, we can readily observe that $2$-dimensional Hamming graphs are the most ``secure'' networks  with respect to active attacks to their privacy, among those ones we have considered so far. This property is possibly related to the high symmetry of $2$-dimensional Hamming graphs as well to the small diameter (of value only two) of such graphs.
\section{Integer programming formulation for the $k$-MAD problem}\label{sect:k_ILP}

In this section we develop an integer programming formulation for finding the $k$-metric antidimension of a given graph $G$. Abusing slightly the notation, we also denote by $V=\{1, \dots, n\}$ the set of indices associated with the vertices of the graph in a natural way. The formulation, which is based on the definition of a $k$-ARS $S$, through the equivalence relation $\mathcal{R}_S$, is built over two sets of binary decision variables, one to determine the elements of the set $S$ and another one to determine the elements of the different classes.
The vertex classes are determined by subsets of vertices that jointly satisfy some compatibility conditions, and will be referred to as $Q$-\emph{subsets}. To avoid multiple representations of the same solution, each $Q$-subset has a unique representative, which is its lowest index vertex.

We define the following sets of decision variables:
\begin{align}
s_u & = 1 \qquad \text{ if and only if } u \in S, \qquad && \forall u\in V \nonumber \\
q_{uv} & = 1 \qquad\text{ if and only if } v\text{ is in the $Q$-subset with representative $u$} \qquad && \forall u,v\in V, v\geq u.\nonumber
\end{align}

The above decision variables determine vertex sets $S=\{u \in V: s_u=1\}$ of cardinality $|S|=\sum_{u \in V} s_u$, and $Q$-subsets $Q^u=\{v\in V: q_{uv}=1\}$. Since $q_{uu}=1$ indicates that $u$ is the lowest index vertex of a $Q$-subset, the actual number of classes determined by the solution is $\overline r=\sum_{u\in V}q_{uu}$.

The formulation is as follows.

\begin{align}
F\qquad\hspace{-0.3cm}\min \quad &  \sum_{u \in V} s_u  && \\
& \sum_{u\in V} s_{u}  \ge 1 && \label{2ineq:S-non-empty}\\
& s_u+\sum_{v\in V: v\leq u} q_{vu}  = 1 && u\in V\label{2ineq:partition}\\ 
& \sum_{v\in V: v> u} q_{uv}  \ge (k-1) q_{uu} && u \in V \label{2ineq:cardinality}\\ 
& s_u+\,q_{vw}  \leq 1  && u, v, w\in V,\, u\ne v,\, v<w,\, d_{uv}\ne d_{uw} \label{2ineq:incompatibility}\\ 
& s_{v}+q_{uv}+\sum_{\substack{w\in V:w\ne u,\, w\ne v, \\ d_{uw}\ne d_{vw}}}s_w \geq q_{uu} \qquad && u,v\in V, v>u \label{2ineq:maximal}\\ 
& s_u\in\{0, 1\}, \, \forall u\in V; && q_{uv}\in\{0, 1\}, u, v \in V, v\geq u.
\end{align}

Let $\Omega$ denote the domain determined by the feasible solutions to $F$, {\em i.e.} $\Omega=\{(s, q): s\in\{0, 1\}^{n}, q \in\{0, 1\}^{\binom{n}{2}} \text{ satisfy } \eqref{2ineq:S-non-empty}-\eqref{2ineq:maximal}\}$.

\begin{prop}\label{propo1}
Any feasible solution $(s, q)\in\Omega$ determines a $k$-{ARS}.
\end{prop}

\begin{proof}
Let $(s, q)\in\Omega$. Consider $S=\{u \in V:  s_u=1\}$ and the $Q$-subsets $Q^u=\{v\in V: q_{uv}=1\}$. Let us see that these $Q$-subsets are precisely the classes determined by the equivalence relation $\mathcal R_{S}$. To this end, we analyze the meaning of the constraints:
\begin{itemize}
\item Constraint \eqref{2ineq:S-non-empty} guarantees that the vertex set $S$ is non-empty. 
\item Constraints \eqref{2ineq:partition} ensure that we obtain a partition by imposing that each vertex $u \in V$ belongs either to set $S$ ($s_u=1$) or to some $Q$-subset, where the representative is either vertex $u$ or a vertex with a lower index. 
\item Constraints \eqref{2ineq:cardinality} guarantee that the cardinality of each $Q$-subset is at least $k$. Note that these constraints are only active when $u$ is the representative of some $Q$-subset ($q_{uu}=1$) and impose that, in such a case, the $Q$-subset associated with $u$ has at least $k-1$ additional elements. This guarantees that $|Q^u|\geq k$.
\item Constraints \eqref{2ineq:incompatibility} ensure that the $Q$-subsets are well defined so the pairs of vertices in each component satisfy the \emph{compatibility}  criterion, by imposing that no two vertices $v, w\in V$ with different distance to any vertex $u\in S$ may belong to the same $Q$-subset. When $w=u$ these constraints impose that the $Q$-sets only contain vertices that are not in $S$. 
\item The role of Constraints~\eqref{2ineq:maximal} is to guarantee that the obtained $Q$-subsets are precisely the classes determined by the equivalence relation $\mathcal R_{S}$. In other words, they guarantee that all the vertices in the same equivalence class of $\mathcal R_{S}$ are assigned to the same $Q$-subset. Note that Constraints \eqref{2ineq:incompatibility} do not guarantee this condition since, in principle, two vertices of the same equivalence class of $\mathcal R_{S}$ could be assigned to different $Q$-subsets.
Observe that the constraint \eqref{2ineq:maximal} associated with a given vertex pair $u,v\in V, u< v$ is only active when $u$ is the representative of $Q$-subset $Q^u$. In such a case, the constraint holds trivially if $v$ belongs either to $S$ or to $Q^u$. Otherwise, the constraint imposes that there exists some vertex $w\in S$ such  that $d_{uw}\ne d_{vw}$. That is, $u$ and $v$ do not belong to the same the equivalence of $\mathcal R_{S}$.\\
\end{itemize}
The conclusion of the above analysis is that any feasible solution $(s, q)\in\Omega$ determines a $k$-ARS and its objective function value $|S|=\sum_{u \in V}  s_u$.
\end{proof}

Below we see that the reverse of the above result also holds. In particular

\begin{prop}\label{propo2}
Any $k$-ARS $S\subset V$, $\{Z^r\}_{r\in R}$ can be associated with a solution in $(s, q)\in\Omega$.
\end{prop}

\begin{proof}
Let $S\subset V$, $\{Z^r\}_{r\in R}$ be a given $k$-ARS. Consider the following solution $(s, q)$:
  \begin{itemize}\item $s_u=1$ if and only if $u\in S$.
  \item $q_{uv}=1$ if and only if $u,v\in Z^r$, $u<v$, and $ q_{u^r, u^r}=1$ with $u^r=\min\{u: u\in Z^r\}$, $r\in R$. \\
  (that is, $Q^{u^r}=Z^r$, $r\in R$).
\end{itemize}
Let us see that $(s, q)\in \Omega$:
\begin{itemize}
\item $S\ne\emptyset$ implies that there exists $u\in S$, and thus   \eqref{2ineq:S-non-empty} holds.
    \item Since $S\cup\{ Z^r\}_{r\in R}$ determines a partition of $V$, by definition, $(s, q)$ satisfies \eqref{2ineq:partition}. 
\item Since $S\subset V$, $\{Z^r\}_{r\in R}$ is a $k$-ARS, for any class $r\in R$ it holds that for any pair of vertices $v,w\in V$ such that $d_{uv}\ne d_{uw}$ for some $u\in S$, then $v$ and $w$ cannot both belong to the same equivalence class $Z^r= Q^{u^r}$, $r\in R$. In other words, if $s_u=1$, then $q_{v,w}=0$ for all $r\in R$. Hence, $(s, q)$ also satisfies constraints \eqref{2ineq:incompatibility}.
\item Since $|Z^r|\geq k$, for all $r\in R$, $|Z^r\setminus \{u^r\}|=|Q^{u^r}\setminus \{u^r\}|\geq k-1$, for all $r\in R$. Hence the constraints \eqref{2ineq:cardinality}, which are only activated for the vertices $u^r$, $r\in R$,  are satisfied by  $(s, q)$.
\end{itemize}
\end{proof}

\begin{rem}
Note that Constraints \eqref{2ineq:S-non-empty} in formulation $F$ impose that the vertex set  $S$ induced by the solution is non-empty. This means that the optimal value of $F$ will be $\adim_k(G)$, when a $k$-ARS exists. However, when no $k$-ARS exists for a given value of $k$, then the formulation $F$ will have no feasible solution.
\end{rem}

As a consequence of the above analysis, we deduce the following statement.

\begin{cor}
$F$ is a valid formulation for \eqref{$k$-MAD Problem} on a given graph $G=(V, E)$. 
When $F$ is feasible, then its optimal value determines $\adim_k(G)$. Otherwise, if $F$ is infeasible, then no $k$-ARS exists.
\end{cor}

\begin{proof}
As a consequence of Propositions \ref{propo1} and \ref{propo2}, there is a one-to-one correspondence between $k$-ARS and feasible solutions in $\Omega$. Moreover, the  objective function value of feasible solution $(s, q)\in\Omega$, is $|S|=\sum_{u \in V}  s_u$. Therefore, since $F$ is a minimization problem, any optimal solution to $F$ will determine a $k$-ARS and its objective function value $\adim_k(G)$.
\end{proof}

The formulation $F$ above has $n+ {n \choose 2}$ binary decision variables, and a number of constraints $\mathcal{O}\left(1+2n+\left(n+1\right){n \choose 2}\right)$.
In particular, the number of Constraints~\eqref{2ineq:incompatibility} is $(n-2){n \choose 2}$. Since this number can be too big as the number of vertices of the graph increases, we develop an aggregated version of this set of constraints, namely:
\begin{align}
\sum_{u\in V: d_{uv}\ne d_{uw}}s_u+ n\,q_{vw} & \leq n && \forall v,w \in V, v\geq w, \label{2ineq:Classes_agg}
\end{align}
which takes into account that in any feasible solution $S$, for any pair $v,w\in V$, $v< w$,  it holds that $\sum_{u\in V: d_{uv}\ne d_{uw}}s_u\le \sum_{u\in V}s_u\le n$. 
Since the right hand side of Constraints \eqref{2ineq:Classes_agg} is precisely $n$, the constraint associated with a given pair of vertices $v, w$ prevents that both vertices belong to the same $Q$-set when there exists some vertex $u\in S$ such that $d_{uv}\ne d_{uw}$. 
Hence by substituting Constraints \eqref{2ineq:incompatibility} with Constraints \eqref{2ineq:Classes_agg} we obtain an alternative valid formulation for \eqref{$k$-MAD Problem}, which will be referred to as $F^{A}$.
Note that the number of Constraints~\eqref{2ineq:Classes_agg} is ${n^2 \choose 2}$, which is one order of magnitude less than that of  Constraints \eqref{2ineq:incompatibility}.

\section{Computational results}\label{sect:compu}

In order to analyze the empirical performance of formulation $F^{A}$ we have carried out a series of computational experiments. The objective of these experiments is twofold. On the one hand, to analyze the effectiveness and scalability of $F^{A}$ for different classes of graphs, and on the other hand, to serve as an empirical support for the classes of graphs for which theoretical results are not known. 

All the computational tests have been carried out in an AMD Ryzen 7 PRO 2700U 2.20 GHz with 8 GB RAM, under Windows 10 Pro as operating system. 
Formulation $F^{A}$  has  been coded in Mosel 5.6.0 using as solver Xpress Optimizer Version~38.01.01 \cite{Xpress}.

For the experiments we have considered the following sets of benchmark instances:
\begin{itemize}
\item $\mathcal{PC}$ instances: Cylinders $P_r\Box C_s$ for combinations \\ $(r, s)\in\{(5, 5),(5, 6),(5, 9),(5, 10),(6, 6),(8, 10),(9, 9),(10, 10)\}$.
\item $\mathcal{CC}$ instances: Torus $C_r\Box C_s$ for combinations \\ $(r, s)\in\{(5, 5),(5, 6),(5, 9),(5, 10),(6, 6),(8, 10),(9, 9),(10, 10)\}$.
\item $\mathcal{T}$ instances: Trees with a number of vertices $n\in\{50, 100, 200\}$, and maximum vertex degree $\delta$, where $\delta\in\{$5, 6, 7, 8, 10, 11, 15, 20$\}$ for $n\in\{50$, 200\} and $\delta\in\{5, 10, 15, 20, 25\}$ for $n=100$. Originally, the trees are generated as directed and rooted at vertex 1.
Then all directions are removed to obtain the resulting undirected tree.
The arcs of the original rooted tree are generated by iteratively \emph{exploring} its vertices and randomly generating up to $\delta$ descendants, among the vertices not yet explored. For each combination of $n$ and $\delta$ two instances have been generated.
\item $\mathcal{S}$ instances: General sparse graphs with a number of vertices $n\in\{50, 100, 200\}$ and vertex degrees $\delta\in\{6, 11\}$.
Similarly to the case of the trees, originally directed graphs are generated and then all directions removed.
The arcs of the graph are generated by iteratively \emph{exploring} its vertices and randomly generating up to $\delta$ end-nodes (from the original vertex set) for the arcs with origin at the current vertex. For each combination of $n$ and $\delta$ two instances have been generated.
\item $\mathcal{D}$ instances: General dense graphs with a number of vertices $n\in\{50, 100, 200\}$, vertex degrees $\delta\in\{40, 45\}$ for $n=50$, $\delta\in\{90, 95\}$ for $n=100$ and $\delta\in\{180, 190\}$ for $n=200$. For each combination of $n$ and $\delta$ two instances have been generated by removing $\delta$ randomly generated edges from the complete graph $K_n$.
\end{itemize}
Cylinder and torus instances have been solved for values of $k\in\{1, 2, 3, 4\}$, whereas instances in the other classes has been solved for values of $k\in\{1, 2, 3, 4, 5, 6\}$. A computing time limit of 7200 seconds has been set for each solved instance.\\

Table \ref{Tab:0} shows the number of variables and constraints in formulation $F^A$  for the considered benchmark instances.

 \begin{table}[ht]
\centering
\small
\begin{center}
\begin{tabular}{|c|cc|cc|c||c|c|cc|c|}												
\hline
\multirow{2}{*}{} & \multirow{2}{*}{$r$} & \multirow{2}{*}{$s$} & \multicolumn{2}{c|}{Vars.}  & \multirow{2}{*}{Constr.}  & \multirow{2}{*}{} & \multirow{2}{*}{$n$}& \multicolumn{2}{c|}{Vars.}  & \multirow{2}{*}{Constr.}\\
\cline{4-5}	\cline{9-10}
 &&&	\multicolumn{1}{c}{$s$} & \multicolumn{1}{c|}{$q$} & &   &&	 \multicolumn{1}{c}{$s$} & \multicolumn{1}{c|}{$q$} & \\	
 \hline																
\multirow{8}{*}{$\mathcal{PC}$ \& $\mathcal{CC}$}&		\multirow{5}{*}{5}	&	5	&	25	&	300	&	1251	&			\multirow{3}{*}{$\mathcal{T}, \mathcal{S}$, \& $\mathcal{D}$}	&	 50	&	 50	&	1225	&	5001\\
	&	  	&	 6	&	 30	&	 435	&	  1801	&		&		100	&	100	&	4950	&	20101\\
	&	  	&	 9	&	 45	&	 990	&	 4051	&		&		200	&	200	&	19900	&	80001\\
\cline{7-11}	&	  	&	10	&	 50	&	1225	&	 5001	&\multicolumn{1}{c}{}\\									
	&	 6	&	 6	&	 36	&	 630	&	 2593	&\multicolumn{1}{c}{}\\								
	&	 8	&	10	&	 80	&	3160	&	 12801	&\multicolumn{1}{c}{}\\									
	&	 9	&	 9	&	 81	&	3240	&	 13123	&\multicolumn{1}{c}{}\\									
	&	10	&	10	&	100	&	4950	&	20001	&\multicolumn{1}{c}{}\\									
\cline{1-6}
\end{tabular}
\caption{Instances dimensions}\label{Tab:0}
\end{center}
\end{table}

Tables \ref{Tab:1}-\ref{Tab:5} summarize the obtained results for each of the classes above. 
All tables show the values of the instance parameters in the first columns, which are followed by $k$ blocks, with two columns each, corresponding to the considered values of $k$. 
The first column in each block, labeled with $|S|$ gives the optimal values for the $k$-metric antidimension of the instances; an entry \emph{NF} (Not Feasible)  
in this column indicates that no $k$-ARS exists for the tested value of $k$; that is, $\adim_k(G)=+\infty$. 
For the classes where two instances have been generated for each combination of parameters values (trees and sparse and dense graphs), the entries in columns $|S|$ show only one value when the optimal value coincided for both instances, whereas  the two (different) optimal values are indicated otherwise. 
The second column in each block, labeled with $\emph{CPU}$ gives the computing time required by the solver to obtain a provable optimal solution, or to show that no feasible solution exists. In Tables \ref{Tab:3}-\ref{Tab:5} the results in these columns are the averages over the two instances with the same characteristics. 
For classes $\mathcal{S}$ \& $\mathcal{D}$ some of the largest instances with $n=200$ could not be solved to proven optimality within the computing time limit. 
For these instances (identified with $TL$ in column \emph{CPU}) the value of the best found solution is presented in the tables.

\begin{table}[ht]
\centering
\small
\begin{center}
\begin{tabular}{|cc|cr|cr|cr|cr|}												
\hline
\multirow{2}{*}{$r$} & \multirow{2}{*}{$s$} & \multicolumn{2}{c|}{$k=1$}  & \multicolumn{2}{c|}{$k=2$}  & \multicolumn{2}{c|}{$k=3$} & \multicolumn{2}{c|}{$k=4$}\\
\cline{3-10}												
 &&	\multicolumn{1}{c}{$|S|$} & \multicolumn{1}{r|}{CPU} &	\multicolumn{1}{c}{$|S|$} & \multicolumn{1}{r|}{CPU}&	\multicolumn{1}{c}{$|S|$} &	 \multicolumn{1}{r|}{CPU}&	 \multicolumn{1}{c}{$|S|$} &	 \multicolumn{1}{r|}{CPU}\\	
 \hline
\multirow{4}{*}{5}	&	 5	&	2	&	0.9	&	1	&	0.2	&	NF	&	0.1	&	 1	&	0.2	\\
                  	&	 6	&	1	&	0.3	&	1	&	0.3	&	 2	&	0.2	&	NF	&	0.2	\\
                  	&	 9	&	2	&	8.4	&	1	&	1.8	&	NF	&	1.6	&	 1	&	1.8	\\
                  	&	10	&	1	&	2.4	&	1	&	2.3	&	 2	&	2.6	&	NF	&	2.8	\\
\hline
               6	&	 6	&	1	&	  0.6	&	4	&	   3.7	&	 2	&	  0.5	&	NF	&	  0.4	\\
               8	&	10	&	1	&	 20.9	&	4	&	 423.9	&	 2	&	 87.7	&	NF	&	223.8	\\
               9	&	 9	&	2	&	146.0	&	1	&	  23.1	&	NF	&	239.8	&	 1	&	 34.6	\\
              10	&	10	&	1	&	 68.5	&	4	&	1748.1	&	 2	&	217.3	&	NF	&	771.5	\\
\hline
\end{tabular}
\caption{Summary of results for $\mathcal{PC}$ instances (cylinders $P_r\Box C_s$)}\label{Tab:1}
\end{center}
\end{table}

\begin{table}[ht]
\centering
\small
\begin{center}
\begin{tabular}{|cc|cr|cr|cr|cr|}												
\hline
\multirow{2}{*}{$r$} & \multirow{2}{*}{$s$} & \multicolumn{2}{c|}{$k=1$} & \multicolumn{2}{c|}{$k=2$} & \multicolumn{2}{c|}{$k=3$}  & \multicolumn{2}{c|}{$k=4$}\\
\cline{3-10}												
 &&	\multicolumn{1}{|c}{$|S|$} & \multicolumn{1}{r|}{CPU} &	\multicolumn{1}{c}{$|S|$} & \multicolumn{1}{r|}{CPU}&	\multicolumn{1}{c}{$|S|$} &	 \multicolumn{1}{r|}{CPU}&	 \multicolumn{1}{c}{$|S|$} &	 \multicolumn{1}{r|}{CPU}\\	
 \hline
\multirow{4}{*}{5}	&	 5	&	2	&	 2.5	&	5	&	 5.1	&	NF	&	 4.1	&	 1	&	0.2	\\
                  	&	 6	&	2	&	 2.1	&	1	&	 0.3	&	NF	&	 0.2	&	NF	&	0.2	\\
                  	&	 9	&	2	&	18.6	&	5	&	45.1	&	NF	&	46.1	&	 1	&	2.4	\\
                  	&	10	&	2	&	10.5	&	1	&	 2.8	&	NF	&	 2.3	&	NF	&	2.8	\\
\hline
 6	&	 6	&	1	&	  0.6	&	4	&	   2.4	&	 4	&	   0.5	&	2	&	 0.7	\\
 8	&	10	&	1	&	 20.9	&	4	&	 742.9	&	 4	&	 102.4	&	2	&	98.4	\\
 9	&	 9	&	2	&	174.9	&	9	&	 933.6	&	NF	&	 371.3	&	1	&	36.2	\\
10	&	10	&	1	&	 66.7	&	4	&	3667.6	&	 4	&	4760.6	&	2	&	131.6	\\
\hline
\end{tabular}
\caption{Summary of results for $\mathcal{CC}$ instances (torus $C_r\Box C_s$)}\label{Tab:2}
\end{center}
\end{table}

As it could be expected, the results in Tables \ref{Tab:1}-\ref{Tab:2} confirm the theoretical results proved in Theorems \ref{th:k_antidim_cyl} and \ref{th:k_antidim_toru}, and the validity of formulation $F^{A}$. In general, for fixed parameters values, cylinder instances $\mathcal{PC}$  can be solved in smaller computing times than torus instances $\mathcal{CC}$, although differences are rather small and there are a few exceptions. 
Indeed the dimensions of the instances affect computing times, which notably increase with the values $(r, s)$. 
For a fixed dimension, the value of parameter $k$ does not seem to  noticeable affect computing times. On the contrary, instances with higher optimal values (higher values of $|S|$) as well as infeasible instances seem to be computationally more demanding.

\addtolength{\tabcolsep}{-2pt}
\begin{table}[ht]
\centering
\scriptsize
\begin{center}
\begin{tabular}{|cc|cr|cr|cr|cr|cr|cr|}												
\hline
\multirow{2}{*}{$n$} & \multirow{2}{*}{$\delta$} & \multicolumn{2}{c|}{$k=1$}  & \multicolumn{2}{c|}{$k=2$}  & \multicolumn{2}{c|}{$k=3$}  & \multicolumn{2}{c|}{$k=4$} & \multicolumn{2}{c|}{$k=5$} & \multicolumn{2}{c|}{$k=6$}\\
\cline{3-14}												
 &&	\multicolumn{1}{c}{$|S|$} & \multicolumn{1}{r|}{CPU} &	\multicolumn{1}{c}{$|S|$} & \multicolumn{1}{r|}{CPU}&	 \multicolumn{1}{c}{$|S|$} &	 \multicolumn{1}{r|}{CPU}&	 \multicolumn{1}{c}{$|S|$} &	 \multicolumn{1}{r|}{CPU} &	 \multicolumn{1}{c}{S} &	 \multicolumn{1}{r|}{CPU} &	 \multicolumn{1}{c}{S} &	 \multicolumn{1}{r|}{CPU}\\	
 \hline
\hline
\multirow{8}{*}{50}	&	 6	&	1	&	1.9	&	1   	&	  2.9	&	1, 8	&	  2.3	&	1, 17	&	2.6	&	 1, 29	&	4.1	&	NF        &	1.9     \\								
                   	&	 7	&	1	&	2.2	&	1   	&	  3.7	&	1, 8	&	  4.8	&	1, 11	&	3.6	&	 1, 15	&	6.6	&	1, 44     &	3.3     \\
                   	&	 8	&	1	&	2.3	&	1   	&	  6.8	&	1   	&	  3.8	&	1    	&	3.6	&	 1    	&	5.8	&	31, 44    &	4.3     \\
                   	&	 9	&	1	&	1.7	&	1   	&	  1.8	&	1   	&	  2.4	&	1, 6 	&	4.3	&	 1, 15	&	4.0	&	4, 16     &	7.0     \\
                   	&	10	&	1	&	1.6	&	1   	&	  2.1	&	1, 3	&	  2.8	&	1, 7 	&	4.0	&	 1, 6 	&	4.5	&	3, 19     &	2.8     \\
                   	&	11	&	1	&	1.8	&	1   	&	  1.7	&	1   	&	  2.6	&	1, 8 	&	4.1	&	 1, 12	&	4.1	&	34, 11    &	3.8     \\
                   	&	16	&	1	&	1.6	&	1   	&	  3.7	&	1, 5	&	  5.6	&	1, 4 	&	2.9	&	 3, 3 	&	5.4	&	1, 2      &	3.8     \\
                   	&	21	&	1	&	1.3	&	1, 5	&	  4.1	&	1, 4	&	  7.4	&	3, 4 	&	6.5	&	 2, 3 	&	4.8	&	1         &	3.2     \\
\hline
\multirow{5}{*}{100}&	 6	&	1	&	56.7	&	1   	&	 76.1	&	1   	&	107.1	&	1   	&	98.0	& 	1    	&	132.0	&	NF   	&	187.3  \\
                   	&	11	&	1	&	38.4	&	1   	&	 41.7	&	1   	&	 51.2	&	1   	&	72.4	& 	1, 22	&	108.4	&	3, 28	&	116.1   \\
                   	&	15	&	1	&	39.2	&	1   	&	 65.0	&	1   	&	 71.8	&	1, 2	&	78.5	& 	1    	&	122.3	&	1    	&	106.8   \\
                   	&	21	&	1	&	35.5	&	1, 3	&	100.7	&	1, 2	&	 93.7	&	1, 3	&	66.6	& 	1    	&	 77.7	&	1    	&	 89.2   \\
                   	&	26	&	1	&	24.5	&	2, 6	&	124.9	&	1, 5	&	 89.0	&	1, 4	&	69.5	& 	1, 3 	&	 54.8	&	1, 2 	&	 43.6   \\
\hline
\multirow{9}{*}{200}&	 6	&	1	&	2553.0	&	1   	&	3105.3	&	1   	&	3321.2	&	1   	&	4244.3	&	1    	&	4266.2		&	1    	&	3009.4     \\
                   	&	 7	&	1	&	2504.5	&	1   	&	2996.5	&	1   	&	3387.3	&	1   	&	3871.3	&	1    	&	4388.2		&	1, 3 	&	4498.7     \\
                   	&	 8	&	1	&	2186.9	&	1   	&	2332.7	&	1   	&	3641.4	&	1   	&	3173.5	&	1, 19	&	4972.6		&	1, 44	&	4208.0     \\
                   	&	 9	&	1	&	2117.2	&	1   	&	2254.3	&	1   	&	3029.4	&	1   	&	3089.9	&	1    	&	4402.9		&	1    	&	4112.6     \\
                   	&	10	&	1	&	2149.2	&	1   	&	2644.2	&	1   	&	2572.8	&	1   	&	3251.9	&	1    	&	3418.3		&	1    	&	3596.4     \\
                   	&	11	&	1	&	2240.4	&	1   	&	2828.9	&	1   	&	2429.5	&	1   	&	4479.9	&	1    	&	3774.2		&	1, 2 	&	3873.7     \\
                   	&	16	&	1	&	1290.5	&	2, 3	&	3366.2	&	1, 2	&	2052.0	&	1   	&	1935.2	&	1    	&	1716.2		&	1    	&	2334.9     \\
                   	&	21	&	1	&	1679.0	&	1, 3	&	3080.6	&	1, 2	&	2960.4	&	1, 2	&	2857.9	&	1    	&	2845.5		&	1    	&	2426.5     \\
                   	&	26	&	1	&	1124.2	&	1, 2	&	2305.0	&	1   	&	1840.1	&	3   	&	3270.9	&	2    	&	1720.9		&	1    	&	2199.7     \\
\hline
\end{tabular}
\caption{Summary of results for $\mathcal{T}$ instances: trees with $n$ vertices and vertex degree $\delta$.}\label{Tab:3}
\end{center}
\end{table}
\addtolength{\tabcolsep}{2pt}

The results of Table \ref{Tab:3} indicate that all but four $\mathcal{T}$ instances (trees) were feasible. For $k=1$ the optimal value was always $|S|=1$, which was also the optimal value for most instances with $k>1$. There are indeed exceptions, particularly for small instances as the value of $k$ increases and gets close to $\delta$. Moreover, for $\delta=k=6$, some $\mathcal{T}$ instances were proved to be infeasible, namely the two 50-vertex instances as well as the two 100-vertex instances. Similarly to $\mathcal{PC}$ and $\mathcal{CC}$ instances, the computing times increase with the number of vertices in the graph, but do not seem to be particularly affected by the value of the parameter $k$. Taking into account the number of variables and constraints involved in the formulation, instances were solved in reasonable computing times. In particular, all instances with $n\in\{50, 100\}$ were solved in less than one hour of computing time. Still, the 200 vertex instances where computationally more demanding, particularly for larger values of $k$, although all of them could be solved within the maximum time limit of two hours.\\

As can be seen in Table \ref{Tab:4}, all $\mathcal{S}$ instances (general sparse graphs) with $n\leq 100$ were feasible, and most of them had an optimal value $|S|=1$; when $|S|>1$, the optimal value of such  instances was $|S|=2$, even when the degree of the vertices is close to the value of $k$. This seems to indicate that, for sparse graphs, allowing cycles reduces the influence of parameter $k$, and facilitates the existence of feasible solutions, at least when $n\leq 100$. However, for five instances with $n=200$ no feasible solution was found. In four of these cases the time limit was reached, so it is not known whether no feasible solution exists for these instances, although this seems unlikely since most of the times a feasible solution (possibly not an optimal one) is found early in the optimization process. For the fifth instance (with parameters $\delta=25$ and $k=6$) the optimization process terminated with $S=\emptyset$ thus proving that no feasible solution exists for that instance.
In general, computing times are higher to those of $\mathcal{T}$ instances with the same parameters values, and the increase becomes more evident as the number of vertices raises. In particular, nine instances with $n=200$ could not be solved to proven optimality within the limit of two hours. While a feasible solution was found within the time limit for five of these instances, no feasible solution was obtained for the other four instances within the allowed computing time.
Table \ref{Tab:5} shows a similar behavior for $\mathcal{D}$ instances (general dense graphs). All but two instances were feasible, and the time limit was reached in the two cases where no feasible solution could be found. It can now be observed that, even if in some cases $|S|=1$, most instances had an optimal value of $|S|=2$, and in some cases $|S|=3$. 
The computing times required to solve $\mathcal{D}$  instances are in the same range as those of $\mathcal{S}$. 
This can be explained by the fact that the number of variables and constraints of formulation $F^A$ depend on the number of vertices of the graph, but do not depend on the number of edges.

\addtolength{\tabcolsep}{-2pt}
\begin{table}[ht]
\centering
\scriptsize
\begin{center}
\begin{tabular}{|cc|cr|cr|cr|cr|cr|cr|}												
\hline
\multirow{2}{*}{$n$} & \multirow{2}{*}{$\delta$} & \multicolumn{2}{c|}{$k=1$}  & \multicolumn{2}{c|}{$k=2$}  & \multicolumn{2}{c|}{$k=3$}  & \multicolumn{2}{c|}{$k=4$} & \multicolumn{2}{c|}{$k=5$} & \multicolumn{2}{c|}{$k=6$}\\
\cline{3-14}												
 &&	\multicolumn{1}{c}{$|S|$} & \multicolumn{1}{r|}{CPU} &	\multicolumn{1}{c}{$|S|$} & \multicolumn{1}{r|}{CPU}&	 \multicolumn{1}{c}{$|S|$} &	 \multicolumn{1}{r|}{CPU}&	 \multicolumn{1}{c}{$|S|$} &	 \multicolumn{1}{r|}{CPU} &	 \multicolumn{1}{c}{S} &	 \multicolumn{1}{r|}{CPU} &	 \multicolumn{1}{c}{S} &	 \multicolumn{1}{r|}{CPU}\\	
 \hline
\multirow{2}{*}{50} &	 6	&	1   	&	  10.5	&	1	&	   1.6	&	1   	&	   6.6	&	1   	&	   1.9	&	1   	&	   1.5	&	1	&	   0.9	 \\
                    &	11	&	1   	&	  14.0	&	1	&	   4.6	&	1   	&	   0.9	&	1   	&	   4.6	&	1   	&	   1.1	&	1	&	   1.0  \\
\hline
\multirow{2}{*}{100}&	 6	&	1   	&	 130.60	&	1	&	  75.17	&	1   	&	 116.3	&	1, 2	&	 230.4	&	1   	&	  51.6	&	1	&	 136.9	 \\
                    &	11	&	2   	&	 202.3	&	2	&	 212.27	&	2   	&	 197.9	&	2   	&	 171.3	&	2   	&	 177.4	&	2	&	 640.8  \\
                    &	16	&	1, 2	&	 125.4	&	1	&	  90.0	&	1   	&	  75.7	&	1   	&	  76.7	&	1   	&	  21.0	&	1	&	  55.8  \\
                    &	21	&	1   	&	 119.7	&	1	&	  43.40	&	1, 2	&	 105.7	&	2   	&	 104.5	&	1, 2	&	  73.3	&	2	&	  74.9  \\
\hline
\multirow{5}{*}{200}	&	 6	&	1, 2	&	3075.6, TL	&	2	&	6481.5    	&	2	&	TL, 5482.6	&	2	&	TL, 6816.6	&	1    	&	1931.7     &	1    	&	3222.6	\\
                    	&	11	&	1   	&	2283.3    	&	1	&	2511.5    	&	1	&	1636.4    	&	1	&	    4519.4	&	1, NF	&	2490.3, TL	&	NF   	&	TL	\\
                    	&	16	&	2   	&	4151.5    	&	2	&	TL, 4658.7	&	2	&	4549.2    	&	2	&	    3520.1	&	2    	&	3958.5     &	2    	&	3793.7	\\
                    	&	21	&	2   	&	3489.5    	&	2	&	2864.1    	&	2	&	3286.7    	&	2	&	    2394.2	&	2    	&	2530.     &	1, NF	&	TL	\\
                    	&	26	&	1   	&	3817.8    	&	1	&	1293.2    	&	1	&	1403.3    	&	1	&	     871.9	&	1    	&	 884.5     &	NF, 1	&	891.1	\\
\hline
\end{tabular}
\caption{Summary of results for $\mathcal{S}$ instances: sparse graphs with $n$ vertices and vertex degree $\delta$.}\label{Tab:4}
\end{center}
\end{table}
\addtolength{\tabcolsep}{2pt}

\addtolength{\tabcolsep}{-2pt}
\begin{table}[ht]
\centering
\scriptsize
\begin{center}
\begin{tabular}{|cc|cr|cr|cr|cr|cr|cr|}												
\hline
\multirow{2}{*}{$n$} & \multirow{2}{*}{$\delta$} & \multicolumn{2}{c|}{$k=1$}  & \multicolumn{2}{c|}{$k=2$}  & \multicolumn{2}{c|}{$k=3$}  & \multicolumn{2}{c|}{$k=4$} & \multicolumn{2}{c|}{$k=5$} & \multicolumn{2}{c|}{$k=6$}\\
\cline{3-14}												
 &&	\multicolumn{1}{c}{$|S|$} & \multicolumn{1}{r|}{CPU} &	\multicolumn{1}{c}{$|S|$} & \multicolumn{1}{r|}{CPU}&	 \multicolumn{1}{c}{$|S|$} &	 \multicolumn{1}{r|}{CPU}&	 \multicolumn{1}{c}{$|S|$} &	 \multicolumn{1}{r|}{CPU} &	 \multicolumn{1}{c}{S} &	 \multicolumn{1}{r|}{CPU} &	 \multicolumn{1}{c}{S} &	 \multicolumn{1}{r|}{CPU}\\	
\hline
\multirow{2}{*}{50} &	 25	&	1	&	   4.1  	&	1   	&	   5.3  	&	1	&	   2.4  	&	1   	&	   1.4  	&	1	&	   1.1  	&	1	&	   2.7   \\
                    &	 30	&	1	&	  26.9  	&	2   	&	  10.3  	&	2	&	  13.4  	&	2   	&	  10.0  	&	2	&	  11.4  	&	2	&	  12.2   \\
                    &	 40	&	3	&	  71.2  	&	2, 3	&	  74.2  	&	2	&	  25.4  	&	2   	&	  26.4  	&	2	&	  16.8  	&	2	&	  20.3   \\
                    &	 45	&	2	&	  22.8  	&	2   	&	  17.8  	&	2	&	  18.7  	&	2   	&	  20.3  	&	2	&	  15.2  	&	2	&	  17.5   \\
\hline
\multirow{2}{*}{100}&	 75	&	2	&	  17.6  	&	2   	&	  12.1  	&	2	&	  14.4  	&	1   	&	   3.6  	&	1	&	   0.9  	&	1	&	   2.9   \\
                    &	 80	&	2	&	2461.8  	&	2, 3	&	1072.3  	&	2	&	1198.8  	&	2, 3	&	 679.6  	&	2	&	 406.6  	&	2	&	 615.2   \\
                    &	 85	&	2	&	 657.4  	&	2   	&	 478.0  	&	2	&	 375.0  	&	2   	&	 249.8  	&	2	&	 219.1  	&	2	&	 202.7   \\
                    &	 90	&	2	&	 320.0  	&	2   	&	 205.4  	&	2	&	 189.0  	&	2   	&	 174.5  	&	2	&	 205.9  	&	2	&	 142.9   \\
                    &	 95	&	2	&	 121.8  	&	2   	&	 205.7  	&	2	&	 156.6  	&	2   	&	 130.5  	&	2	&	 180.4  	&	2	&	 168.9   \\
\hline
\multirow{5}{*}{200}&	175	&	2, 3	&	TL 	&	2	&	4790.1, TL &	2	&	TL, 3715.2 &	2	&	TL, 2961.7	&	2	&	2280.1    	&	 2	&	2088.5  \\
                    &	180	&	2   	&	3553.4	&	2	&	1202.9    	&	2	&	1217.9    	&	2	&	TL, 1095.5	&	2	&	TL, 1239.8	&	 2	&	2065.7	\\
                    &	185	&	2   	&	3553.4	&	2	&	1202.9	   	&	2	&	1217.9	   	&	2	&	TL, 1095.5	&	2	&	TL, 1239.8	&	 2	&	2065.7	\\
                    &	190	&	2   	&	1574.8	&	2	&	1598.3	   	&	2	&	2222.7	   	&	2	&	TL, 1719.1	&	2	&	    2125.3	&	NF	&	TL, TL 	\\
                    &	195	&	2   	&	1201.7	&	2	&	1216.2	   	&	2	&	2759.8	   	&	1	&	    2275.7	&	1	&	    1281.2	&	 1	&	1389.5	\\
\hline
\end{tabular}
\caption{Summary of results for $\mathcal{D}$ instances: dense graphs with $n$ vertices and degree $\delta$.}\label{Tab:5}
\end{center}
\end{table}
\addtolength{\tabcolsep}{2pt}

On the other hand, in order to further analyze \eqref{$k$-ARS Problem}, and the influence of the parameter $k$ on the feasibility and optimal values of instances, we run a final series of experiments in which we solved instances for increasing values of $k$ starting with $k=1$. We now used the $\mathcal{T}$, $\mathcal{S}$ and $\mathcal{D}$ instances with $n=100$. Each of these instances was solved for varying values of the parameter $k$, $1\leq k\leq \min\{\delta, \lfloor\frac{n}{2}\rfloor\}$.  

Table \ref{Tab:barrido arbol} summarizes the obtained results on tree instances $\mathcal{T}$, whereas Tables \ref{Tab:ultima1}-\ref{Tab:ultima2_2} summarize the results on instances $\mathcal{S}$ and $\mathcal{D}$.

\begin{table}[ht]
\centering
\scriptsize
\begin{center}
\begin{tabular}{|c|c|c|c|c|c|}												
\hline
\multirow{2}{*}{$k$} & \multicolumn{1}{c|}{$\delta=6$} & \multicolumn{1}{c|}{$\delta=11$} & \multicolumn{1}{c|}{$\delta=16$} & \multicolumn{1}{c|}{$\delta=21$} & \multicolumn{1}{c|}{$\delta=26$}\\
& \multicolumn{1}{c|}{$|S|$} & \multicolumn{1}{c|}{$|S|$} & \multicolumn{1}{c|}{$|S|$}& \multicolumn{1}{c|}{$|S|$}& \multicolumn{1}{c|}{$|S|$}\\
\hline
 1 &\multirow{5}{*}{$1$}   & \multirow{4}{*}{1}   & \multirow{3}{*}{1}  &       1   &    1   \\
 2 &                       &                      &                     &   3, 1    &   2, 6  \\
 3 &                       &                      &                     &   2, 1    &   1, 5  \\
\cline{4-4}
 4 &                       &                      &  2, 1               & 1, 3      &   1, 4  \\
\cline{3-4}
 5 &                       & 1, 22                & \multirow{4}{*}{1}  &  1        &   1, 3  \\
\cline{2-2}
 6 &                       & 3, 28                &                     &  1        &   2, 1  \\
 7 &                       & 1, 27                &                     &  2, 37    &    1    \\
 8 &                       & 1, 35                &                     &  1, 26    &   1, 19  \\
\cline{4-4}
 9 &                       & 72, 34               & 4, 6                &  1, 36    &  3, 18   \\
10 &                       & 90, 80               & 3, 5                &  59, 35   & 2, 17\\
11 &                       &    NF                & 2, 4                &  72, 34   & 1, 16  \\
 \cline{3-6}
12 &                       &                      &      1              & 71, 45    &  1, 15 \\
13 &                       &                      & 3, 50               & 70, 57    & 63, 27 \\
14 &                       &                      & 2, NF               & 69, 71    & 62, 40 \\
15 &                       &                      & 1, NF               & 85, 70    & 61, 60 \\
 \cline{4-6}
16 &                       &                      &                     & 84, NF    &  60, 59   \\
17 &                       &                      &                     & 83, NF    &  59, 58   \\
18 &                       &                      &                     & NF        &  82, 57   \\
19 &                       &                      &                     & NF        &  81,56     \\
 \cline{5-6}
20 &                       &                      &                     &           &  80, 55    \\
21 &                       &                      &                     &           &  79, 54    \\
22 &                       &                      &                     &           &  78           \\
23 &                       &                      &                     &           &  77           \\
24 &                       &                      &                     &           &  76           \\
25 &                       &                      &                     &           &  NF, 75        \\
\hline
\end{tabular}
\caption{Optimal values for $\mathcal{T}$ instances for varying values of $\delta$ and $k$.}\label{Tab:barrido arbol}
\end{center}
\end{table}

\begin{table}[ht]
\centering
\small
\begin{center}
\begin{tabular}{|cc|cc|cc|cc|cc|cc|}												
\hline
\multicolumn{4}{|c|}{$\delta=11$} & \multicolumn{4}{c|}{$\delta=22$} & \multicolumn{4}{c|}{$\delta=26$}\\
\hline
\multicolumn{2}{|c|}{\text{Instance 1}} & \multicolumn{2}{c|}{\text{Instance 2}} & \multicolumn{2}{c|}{\text{Instance 1}} & \multicolumn{2}{c|}{\text{Instance 2}} & \multicolumn{2}{c|}{\text{Instance 1}} & \multicolumn{2}{c|}{\text{Instance 2}} \\
\multicolumn{1}{|c}{$k$} & \multicolumn{1}{c|}{$|S|$} & \multicolumn{1}{c}{$k$} & \multicolumn{1}{c|}{$|S|$} &\multicolumn{1}{c}{$k$} & \multicolumn{1}{c|}{$|S|$}&\multicolumn{1}{c}{$k$} & \multicolumn{1}{c|}{$|S|$}&\multicolumn{1}{c}{$k$} & \multicolumn{1}{c|}{$|S|$}&\multicolumn{1}{c}{$k$} & \multicolumn{1}{c|}{$|S|$}\\
\hline
 [1, 9] & 2 &	[1, 9]	& 2	&	[1, 9]  	&	 2	&	  [1, 3]	&	 1	&	[1, 2]   &   1 & 	[1, 12]  & 	2	\\
  10    & 1	&	10    	& 1	&	[10, 13]	&	 1	&	       4	&	 2	&	[3, 12]  & 	 2 & 	13       &	1	\\
  $>$10   & NF&	$>$10   	&NF	&	14      	&	NF	&	       5	&	 1	&	[13, 15] & 	 1 & 	[14, 24] &	NF	\\
        &   &	      	&  	&	15      	&	 1	&	  [6, 9]	&	 2	&	16       &  NF & 	[25, 26] &	1	\\
        &   &	      	&  	&	[16, 17]	&	NF	&	[10, 14]	&	 1	&	[17, 21] & 	 1 & 	$>$26      &	NF	\\
        &   &	      	&  	&	18      	&	 1	&	      15	&	NF	&	22       &  NF & 	       &       \\
        &   &	      	&  	&	19      	&	 2	&	      16	&	 1	&	23       & 	 1 & 	       &       \\
        &   &	      	&  	&	20      	&	 1	&	[17, 18]	&	NF	&	24       & 	 2 & 	       &       \\
        &   &	      	&  	&	$>$20     	&	NF	&	      19	&	 2	&	[25, 26] & 	 1 & 	       &       \\
        &   &	      	&  	&	        	&	  	&	      20	&	 1	&	$>$26      &  NF & 	       &         \\
        &   &	      	&  	&	        	&	  	&	     $>$20	&	NF	&	         &     & 	      	&      	\\
 \hline
 \end{tabular}
\caption{Optimal values for $\mathcal{S}$ instances for varying values of $\delta$ and $k$.}\label{Tab:ultima1}
\end{center}
\end{table}

\begin{table}[ht]
\centering
\scriptsize
\begin{center}
\begin{tabular}{|cc|cc|cc|cc|cc|cc|}												
\hline
\multicolumn{4}{|c|}{$\delta=75$} & \multicolumn{4}{c|}{$\delta=80$} & \multicolumn{4}{c|}{$\delta=85$}\\
\hline
\multicolumn{2}{|c|}{\text{Instance 1}} & \multicolumn{2}{c|}{\text{Instance 2}} & \multicolumn{2}{c|}{\text{Instance 1}} & \multicolumn{2}{c|}{\text{Instance 2}} & \multicolumn{2}{c|}{\text{Instance 1}} & \multicolumn{2}{c|}{\text{Instance 2}}\\
\multicolumn{1}{|c}{$k$} & \multicolumn{1}{c|}{$|S|$} & \multicolumn{1}{c}{$k$} & \multicolumn{1}{c|}{$|S|$} &\multicolumn{1}{c}{$k$} & \multicolumn{1}{c|}{$|S|$}&\multicolumn{1}{c}{$k$} & \multicolumn{1}{c|}{$|S|$}&\multicolumn{1}{c}{$k$} & \multicolumn{1}{c|}{$|S|$}&\multicolumn{1}{c}{$k$} & \multicolumn{1}{c|}{$|S|$} \\
												
\hline
      1   &  3  &   [1, 2]	&  3 	&  [1, 18] &  2	&  [1, 19] 	&   2	&  [1, 13] 	&  2 	&  [1, 13] 	&  2 	 \\
   $[2, 3]$ &  2  &  {\tiny[3, 23]}	&  2 	& [19, 21] &  1	& [20, 22] 	&   1	& [14, 32] 	&  1 	& [14, 15] 	&  1   	\\
      4   &  3  &       24	&  1 	&       22 & NF	&       23 	&  NF	&      $>$32 	& NF 	&       16 	&  2 		\\
 $[5, 24]$  &  2  &       25	& NF 	& [23, 27] &  1	& [24, 33] 	&   1	&          	&    	& [17, 33] 	&  1 	\\
$[25, 30]$  &  1  & [26, 28]	&  1 	&       28 & NF	&      $>$33 	&  NF	&&&&\\
      31  & NF  &       29	& NF 	& [29, 33] &  1	&          	&    	&          	&    	&          	&    	\\
$[32, 33]$  &  1  &       30	&  1 	&      $>$33 & NF	&          	&    	&          	&    	&          	&    	\\
$[34, 47]$  &  2  & [31, 32]	& NF 	&          &   	&          	&    	&          	&    	&          	&    	\\
$[48, 50]$  & NF  &       33	&  1 	&          &   	&          	&    	&          	&    	&          	&    	\\
          &     & [34, 45]	&  2 	&          &   	&          	&    	&          	&    	&          	&    	\\
          &     &      $>$45	& NF 	&          &   	&          	&    	&          	&    	&          	&    	\\
 \hline
 \end{tabular}
\caption{Optimal values for $\mathcal{D}$ instances for $\delta\in\{75, 80, 85\}$ and varying values of $k$.}\label{Tab:ultima2}
\end{center}
\end{table}

\begin{table}[ht]
\centering
\scriptsize
\begin{center}
\begin{tabular}{|cc|cc|cc|cc|}												
\hline
\multicolumn{4}{|c|}{$\delta=90$}& \multicolumn{4}{c|}{$\delta=95$}\\
\hline
\multicolumn{2}{|c|}{\text{Instance 1}} & \multicolumn{2}{c|}{\text{Instance 2}} & \multicolumn{2}{c|}{\text{Instance 1}} & \multicolumn{2}{c|}{\text{Instance 2}} \\
\multicolumn{1}{|c}{$k$} & \multicolumn{1}{c|}{$|S|$} & \multicolumn{1}{c}{$k$} & \multicolumn{1}{c|}{$|S|$}& \multicolumn{1}{c}{$k$} & \multicolumn{1}{c|}{$|S|$}& \multicolumn{1}{c}{$k$} & \multicolumn{1}{c|}{$|S|$}\\							
\hline
 [1, 8]	&  2&[1, 8] &   2       &   [1, 3]	    &  2 	& [1, 3]&  2   	 \\
$[9, 22]$	&  1&[9, 23]&   1       &  [4, 11]	    &  1 	&[4, 11]&  1   	\\
$>$22	& NF&  $>$23&  NF       &       12	    & NF 	& $>$11	& NF   	\\
        &   &       &        	&       13	    &  1 	&       &      	\\
  	    &   &       &   		&    [14, 24]	& NF 	&       &      	\\
        & 	&       &   		&       25   	&  4 	&       &      	 \\
    	&	&       &          	&      $>$25	& NF 	&       &      	\\
 \hline
 \end{tabular}
\caption{Optimal values for $\mathcal{D}$ instances $\delta\in\{90, 95\}$ and varying values of $k$.}\label{Tab:ultima2_2}
\end{center}
\end{table}

The results of the implementations appearing in Table \ref{Tab:barrido arbol} show that the existence of $k$-ARSs in a tree is always possible, for every $k$ between $1$ and the maximum value that gives a feasible solution. This confirms the theoretical results obtained in \cite{DasGupta-2019} about the existence of $k$-ARSs in trees. On the other hand, the experiments on general graphs, both sparse and dense, shown in Tables  \ref{Tab:ultima1} and \ref{Tab:ultima2} confirm that, when the graphs are not trees, the existence of $k$-ARSs is not always warranted.  Note that for several values of $k$ smaller than the maximum possible $k$, no feasible solution exits.

\subsection{The $(k, \ell)$-anonymity met by the randomly generated graphs}

The computational results presented in the tables of this section allow to conclude that, in general, graphs randomly generated are usually satisfying a low security with respect to active attacks to its privacy, under the assumption of the existence of one or two attacker vertices ($\ell = 1$ or $\ell = 2$). In particular, we conclude the following.
\begin{itemize}
  \item Trees belonging to $\mathcal{T}$ instances are satisfying only $(1, 1)$-anonymity.
  \item General sparse graphs belonging to $\mathcal{S}$ instances sometimes satisfy $(1, 1)$-anonymity (usually when the graphs have smaller order), and sometimes $(2, 1)$-anonymity (more frequently when the graphs are of larger order).
  \item For the case of general dense graphs belonging to $\mathcal{D}$ instances, there are very few cases such that there is a value of $k$ for which $\adim_k(G)\le 1$. This could indicate that such graphs achieve a higher security with respect to active attacks. For such dense graphs, we observe that they usually satisfy $(1, 2)$-anonymity, i.e., an attacker needs to control at least two vertices in such graphs to have some success.
\end{itemize}

The results of the last item are consistent with the characteristics of dense graphs, which are likely to have a small diameter as well as nearly-symmetrical structures, similarly to $2$-dimensional Hamming graphs, which have been proved to have the highest security properties with respect to their privacy.

\section{Concluding remarks}\label{sect:conclu}

In this research we have considered the problem of evaluating the $(k, \ell)$-anonymity measure for several networks of two different types. 
A first type consists of graphs with a predetermined structure, namely cylinders, toruses, and $2$-dimensional Hamming graphs, and a second one is formed by randomly generated graphs. 
Both types have in common the property of being regular or almost regular graphs. 
In order to proceed with this evaluation, we have studied the $k$-antiresolving sets and the $k$-metric antidimension of the considered networks. The later is a combinatorial parameter, of interest on its own in graph theory.

The main contributions of this research are the following:
\begin{enumerate}
\item Concerning cylinders, toruses and $2$-dimensional Hamming graphs: 
\begin{itemize}
  \item We have found the largest possible value of $k$ for which a $k$-ARS exists.
  \item We have computed the exact value for the $k$-metric antidimension for the values of $k$ such that a $k$-ARS exists,  and have proved when such $k$-ARS does not exist.
  \item We have stated their $(k,\ell)$-anonymity according to the privacy measure under study.
  \end{itemize}
    As a conclusion, we have obtained that cylinders and toruses have a low privacy, in contrast to $2$-dimensional Hamming graphs, which show a higher privacy. This seems to be mainly due to the highly symmetrical structure of such graphs that are also of diameter two.
  \item We have developed an integer programming formulation for finding the $k$-metric antidimension of a given graph $G$. The formulation can be used as a tool to deal with randomly generated graphs as well as with classes of graphs for which their $k$-metric antidimension is not known theoretically. 
  \begin{itemize}
  \item The formulation has been validated by implementing it for the cylinder and torus graphs, previously studied in a theoretical way.
  \item Using the formulation we have computed the $k$-metric antidimension for randomly generated trees and randomly generated general graphs, both sparse and dense. For all generated instances we have solved both \eqref{$k$-MAD Problem} and \eqref{$k$-ARS Problem}.
  \item Using the results of the implementations we have evaluated the $(k,\ell)$-anonymity met by the tested graphs. 
   \end{itemize}
   The obtained results indicate that random trees and general sparse graphs achieve low privacy properties, whereas random general dense graphs exhibit higher privacy properties. Similartly to $2$-dimensional Hamming graphs, this can be due to the fact that dense graphs have more ``near-symmetrical'' properties than the other ones, and also that they intuitively should have small diameters.
\end{enumerate}
As a general conclusion concerning the $(k, \ell)$-anonymity, it seems that higher privacy properties appear in graphs with more symmetry, larger degrees and smaller diameters.\\

Possible avenues for future research are the following:
\begin{itemize}
  \item A general study, from a graph theory point of view, of the $k$-metric antidimension of Cartesian product graphs as well as other related products. This would extend our contributions on cylinders, toruses and $2$-dimensional Hamming graphs.
   \item A deeper study of the relationship between high privacy indicators and symmetry properties of graphs with larger degrees and smaller diameters (with respect to their orders). 
\item Study of the complexity of computing the $k$-metric antidimension of trees. That is, given an arbitrary integer $k$ and a tree $T$, is it polynomial to find $\adim_k(T)$? Notice that the existence of such value is already warranted (see \cite{DasGupta-2019}).
  \item Apply the implementations of the integer programming formulation to real social networks, in order to evaluate their actual privacy features.
\end{itemize}

\section*{Acknowledgements}

The authors have been partially supported by the Spanish Agencia Estatal de Investigación and European Regional Development Funds (ERDF) through MINECO PID project MTM2019-105824GB-I00, as well as by the Plan Propio - UCA 2022-2023. Moreover, this investigation was completed while the second author (Dorota Kuziak) was making a temporary stay at the Rovira i Virgili University supported by the program ``Ayudas para la recualificaci\'on del sistema universitario espa\~{n}ol para 2021-2023, en el marco del Real Decreto 289/2021, de 20 de abril de 2021''.

\section*{Author contributions statement}

All authors contributed equally to this work.

\section*{Additional information}

The authors have no competing interests.

\bibliographystyle{abbrv}

\section*{Appendix}

We next include the proofs of the results presented in Section \ref{sec:Cyl-To-Ham}. The proofs follow somehow a similar structure. We first give conditions that a given set of vertices must satisfy in order to be a $k$-ARS for a given graph. Then, we complete the computations, by constructing a $k$-ARS with the required cardinality that gives the exact value. In the process we make use of the following and simple result.

\begin{rem}{\em\cite{Trujillo-Rasua-2016}}
\label{rem_Delta}
If $G$ is a graph of maximum degree $\Delta$, then it is $k$-metric antidimensional for some $k\le \Delta$.
\end{rem}

\subsection*{Proof of Proposition \ref{prop:k_antidim_toru_cyl_Ham}}

\paragraph{(i)} Consider a cylinder $P_r\Box C_s$ with $r\ge 2$ and $s\ge 3$. Since $P_r\Box C_s$ has maximum degree $4$ (unless $r=2$ when it is $3$-regular), from Remark \ref{rem_Delta}, it must be $k$-metric antidimensional for some $k\le 4$ (or $k\le 3$ when $r=2$). We analyze the following situations.

\begin{itemize}
\item \underline{$r,s$ are odd.}
Notice that the vertex in $S_1=\{(u_{(r+1)/2},v_0)\}$ has four eccentric vertices, which are $(u_1,v_{(s-1)/2})$, $(u_1,v_{(s+1)/2})$, $(u_r,v_{(s-1)/2})$, $(u_r,v_{(s+1)/2})$, and they form an equivalence class in $\mathcal{Z}_{S_1}$.
Notice that the vertex set of $P_r\Box C_s$ can be partitioned into sets $S_1,X_1,\dots, X_{q}$ where $X_i$, $i\in\{1,\dots,q\}$, is formed by those vertices at distance $i$ from the vertex in $S_1$, and $q$ is the eccentricity of $(u_{(r+1)/2},v_0)$.
Moreover, given any vertex, there are at least four disjoint paths, excluding the first vertex, between the vertex $(u_{(r+1)/2},v_0)$ and the vertices $(u_1,v_{(s-1)/2})$, $(u_1,v_{(s+1)/2})$, $(u_r,v_{(s-1)/2})$, and $(u_r,v_{(s+1)/2})$, such that one of those path passes thought such given vertex.
Based on these facts above, it can be observed that any equivalence class in $\mathcal{Z}_{S_1}$ has cardinality at least four (these classes are precisely the sets $X_1,\dots, X_{q}$).
Thus, $S_1$ is a $4$-ARS.
Since $P_r\Box C_s$ has maximum degree four, we deduce that $P_r\Box C_s$ is $4$-metric antidimensional when $r,s$ are odd.

We remark that this last discussion about the size of the classes defined by a central vertex can be extended to any grid or torus.
This fact will be used several times in the following proofs.

\item \underline{$s$ is even.}
If $r=2$, then $P_r\Box C_s$ has maximum degree three. Hence, an argument similar to the one seen above leads to the conclusion that the set $S_2=\{(u_1,v_0),(u_2,v_{s/2})\}$ is a $3$-ARS.
Thus, $P_2\Box C_s$ is $3$-metric antidimensional in this case. From now on, let $r\ge 3$, and let $S_3=\{(u_1,v_0),(u_r,v_{s/2})\}$. It can be observed that $S_3$ is a $3$-ARS, which means $P_r\Box C_s$ is $k$-metric antidimensional for some $k\ge 3$.

Now, suppose that $P_r\Box C_s$ contains a $4$-ARS (it cannot be more because $P_r\Box C_s$ has maximum degree four). We consider two cases.

\medskip
\noindent
\underline{\textbf{Case 1:} $r$ is even.} First observe that any vertex $x = (u_i, v_j)$ of $P_r\Box C_s$ has a unique eccentric vertex.
Thus, if one looks for a $4$-ARS $A$ and $x\in A$, then its unique eccentric vertex must be in $A$ as well.
The eccentric vertex of $(u_i,v_j)$ is either $(u_1,v_{j+s/2})$ or $(u_r,v_{j+s/2})$, say $(u_r,v_{j+s/2})$ (which can be assumed by the symmetry of $P_r\Box C_s$).
Thus, $(u_r,v_{j+s/2})\in A$ too.
But then, since the vertex $(u_r,v_{j+s/2})$ has degree three, we obtain that either the set $A$ cannot be a $4$-ARS, contradicting our assumption, or all the neighbors of $(u_r,v_{j+s/2})$ are in $A$.
This latter situation leads to the conclusion that all the vertices of the copy of $C_s$ corresponding to $u_r$ must be in $A$ as well, which is again a contradiction since such $A$ is not a $4$-ARS.

\medskip
\noindent
\underline{\textbf{Case 2:} $r$ is odd.}
The proof is relatively similar to Case 1, but we must remark that any vertex $x$ of $P_r\Box C_s$ has a unique eccentric vertex, unless $x$ is equal to $(u_{(r+1)/2},v_j)$ for any $j\in \{0,\dots,s-1\}$.
If the vertex $(u_i,v_j)\in A$ considered above satisfies that $i\ne (r+1)/2$, then the argument of Case 1 works in the same way to get the same conclusion.
So, assume $i=(r+1)/2$.
Hence, there are two eccentric vertices of $(u_i,v_j)$, which are $(u_1,v_{j+s/2})$ and $(u_r,v_{j+s/2})$.
Thus, it must happen $(u_1,v_{j+s/2}),(u_r,v_{j+s/2})\in A$ too. But then, since the vertices $(u_r,v_{j+s/2}),(u_1,v_{j+s/2})$ have degree three, either the set $A$ cannot be a $4$-ARS, contradicting the assumption again, or as in Case 1, all the vertices of the copies of $C_s$ corresponding to $u_1,u_r$ must in $A$ as well, which is again a contradiction since such $A$ is not a $4$-ARS.

\medskip
Therefore, both cases above lead to conclude that $P_r\Box C_s$ is $3$-metric antidimensional when $r$ is even.

\item \underline{$s$ is odd and $r$ is even.}
First observe that any single vertex of $P_r\Box C_s$ has two eccentric vertices, and it forms a $2$-ARS.
Thus, $P_r\Box C_s$ is $k$-metric antidimensional for some $k\ge 2$ in this case. Suppose that $P_r\Box C_s$ contains a $k$-ARS for some $k\ge 3$.

Let $S_4$ be $k$-ARS with $k\ge 3$ and let $(u_i,v_j)\in S_4$.
By the reasons stated above, the two eccentric vertices of $(u_i,v_j)$ must be in $S_4$ too.
By the symmetry of $P_r\Box C_s$, we can assume without loss of generality that these vertices are $(u_r,v_{j+(s-1)/2})$ and $(u_r,v_{j+(s+1)/2})$ (this is also based on the fact that $r$ is even).
This immediately means that $S_4$ cannot be a $4$-ARS because the vertices $(u_r,v_{j+(s-1)/2})$ and $(u_r,v_{j+(s+1)/2})$ have degree three, and a similar procedure as in the Cases 1 and 2 above can be applied.
On the other hand, since $(u_r,v_{j+(s-1)/2})$ and $(u_r,v_{j+(s+1)/2})$ are in $S_4$, it must happen that $(u_1,v_{j-1})$, $(u_1,v_{j})$ and $(u_1,v_{j+1})$ are also in $S_4$.
By the same reasons, it must also happen that $(u_r,v_{j+(s-1)/2-1})$ and $(u_r,v_{j+(s+1)/2+1})$ are in $S_4$ as well.
This argument will lead to the conclusion that all the vertices $(u_1,v_\ell)$ and $(u_r,v_\ell)$, with $\ell\in \{0,s-1\}$, must be in $S_4$.
However, such a set is only a $1$-ARS, which is a contradiction.
Consequently, $P_r\Box C_s$ is $2$-metric antidimensional when $r$ is even and $s$ is odd.
\end{itemize}

\paragraph{(ii)} Consider a torus $C_r\Box C_s$ with $r,s\ge 3$. Since $C_r\Box C_s$ is $4$-regular, from Remark \ref{rem_Delta}, it must be $k$-metric antidimensional for some $k\le 4$.

\begin{itemize}
\item If $r,s$ are both even, then consider the two diametral vertices $u_0$ and $u_{r/2}$ of $C_r$, and the two diametral vertices $v_0$ and $v_{s/2}$ of $C_s$.
Notice that $\{u_0,u_{r/2}\}$ forms a $2$-ARS of $C_r$, as well as $\{v_0,v_{s/2}\}$ in $C_s$.
We can readily see that the two vertices $(u_0,v_0)$ and $(u_{r/2},v_{s/2})$ form a $4$-ARS of $C_r\Box C_s$. \item If $r,s$ are both odd, then the vertices $u_0$ and $v_0$ form $2$-ARSs of $C_r$ and $C_s$, respectively. Thus, it can be noticed that the vertex $(u_0,v_0)$ forms a $4$-ARS of $C_r\Box C_s$.

\item It remains to consider the case in which $r,s$ are of different parity. Assume that $r$ is even and $s$ is odd. Thus, for any vertex $(u_i,v_j)\in V(C_r\Box C_s)$, there are exactly two diametral vertices, and so such vertex forms only a $2$-ARS of $C_r\Box C_s$.  Consequently, if one wants to construct a $k$-ARS $S$ of $C_r\Box C_s$ with $k\ge 3$, then we need to consider a set of cardinality larger than one. However, in such situation there will always be an equivalence class $Z^i\in \mathcal{Z}_S$ with cardinality at most $2$, since for any vertex $(u_i,v_j)\in V(C_r\Box C_s)$, there are exactly two diametral vertices in $C_r\Box C_s$. This proves that $C_r\Box C_s$ is $2$-metric antidimensional.
\end{itemize}

\paragraph{(iii)} Consider a $2$-dimensional Hamming graph $K_r\Box K_r$ with $k\geq 4$. 
Notice that $K_r\Box K_r$ is $(2r-2)$-regular. 
Thus, from Remark~\ref{rem_Delta}, it is $k$-metric antidimensional for some $k\le 2r-2$, and it is not difficult to check that any vertex of $K_r\Box K_r$ forms a $(2r-2)$-ARS.
\qed

\subsection*{Proof of Theorem \ref{th:k_antidim_cyl}}
Let us see that for every integers $r\ge 2$ and $s\ge 3$,
$$\adim_k(P_r\Box C_s)=\left\{\begin{array}{ll}
                                1, & \mbox{if $k=4$ and $r,s$ are odd}, \\
                                2, & \mbox{if $k=3$ and $s$ is even}, \\
                                +\infty, & \mbox{if $k=3$ and $r,s$ are odd}, \\
                                1, & \mbox{if $k=2$ and $r,s$ are not both even},\\
                                4, & \mbox{if $k=2$ and $r,s$ are even},\\
                                2, & \mbox{if $k=1$ and $s$ is odd}, \\
                                1, & \mbox{if $k=1$ and $s$ is even}.
                              \end{array}\right.
  $$
We separate our arguments in four cases according to the possible values of $k$ for which we can compute $\adim_k(P_r\Box C_s)$.

\medskip
\noindent
\textbf{\underline{Case $k=4$:}} We only need to consider the situation when $r,s$ are both odd. Since any single vertex $(u_{(r+1)/2},v_j)$ with $j\in\{0,\dots,s-1\}$ forms a $4$-ARS, it is clear that $\adim_4(P_r\Box C_s)=1$.

\medskip
\noindent
\textbf{\underline{Case $k=3$:}} If $s$ is even, then for any vertex $(u_i,v_j)$ of $P_r\Box C_s$, there is exactly one eccentric vertex, unless $r$ is odd and $i=(r+1)/2$, in which case $(u_i,v_j)$ has exactly two eccentric vertices.
Thus, any single vertex forms a $k$-ARS of $P_r\Box C_s$ with $k\le 2$, and so $\adim_3(P_r\Box C_s)\ge 2$.
On the other hand, we consider the set $S_1=\{(u_0,v_1),(u_r,v_{s/2})\}$ (two diametral vertices).
Notice that in $\mathcal{Z}_{S_1}$ there are two classes of cardinality $3$, one formed by the three neighbors of $(u_0,v_1)$ and the other one by the three neighbors of $(u_r,v_{s/2})\}$.
Now, the structure of $P_r\Box C_s$ allows to observe that every vertex of $P_r\Box C_s$ belongs to a diametral path between $(u_0,v_1)$ and $(u_r,v_{s/2})$, and so, vertices having the same distance to $(u_0,v_1)$ also have the same distance to $(u_r,v_{s/2})$.
This means that any other remaining class (if it exists) of $\mathcal{Z}_{S_1}$ has cardinality at least $4$.
Thus, $S_1$ is a $3$-ARS of $P_r\Box C_s$, so $\adim_3(P_r\Box C_s)=2$.

Next, let $r,s$ be odd. As mentioned before, any single vertex $(u_{(r+1)/2},v_j)$ with $j\in\{0,\dots,s-1\}$ forms a $4$-ARS, and any other different vertex forms a $2$-ARS. Thus, if there would be a $3$-ARS, such a set should have cardinality at least $2$. Suppose $S_2$ is such a set. If $S_2\subseteq (\{u_{(r+1)/2}\}\times V(C_s))$, then we readily observe that $S_2$ is not a $3$-ARS. Thus, consider a vertex $(u_i,v_j)\in S_2$ with $j\ne (r+1)/2$. Since $(u_i,v_j)$ has two eccentric vertices which are, without loss of generality, $(u_r,v_{j+(s-1)/2})$ and $(u_r,v_{j+(s+1)/2})$. By using the same argument as in the proof of Proposition \ref{prop:k_antidim_toru_cyl_Ham} (i), we deduce that all the vertices $(u_1,v_\ell)$ and $(u_r,v_\ell)$, with $\ell\in \{0,s-1\}$, must be in $S_2$. However, such set is only a $1$-ARS, which is a contradiction. Therefore, there are not $3$-ARSs in $P_r\Box C_s$ when $r,s$ are odd.

\medskip
\noindent
\textbf{\underline{Case $k=2$:}} If $r,s$ are even, then every vertex has exactly one eccentric vertex, and it forms only a $1$-ARS. Thus, $\adim_2(P_r\Box C_s)\ge 2$. Let $S_3$ be a $2$-ARB.
Notice also that if a vertex $(u_i,v_j)\in S_3$, then its eccentric vertex must be in $S_2$ as well. This means that if $\adim_2(P_r\Box C_s)=2$, then $S_3$ is formed by two diametral vertices of $P_r\Box C_s$. However, as proved for the case $k=3$, such set is a $3$-ARSs in $P_r\Box C_s$, which is not possible. Thus $\adim_2(P_r\Box C_s)\ge 3$. Suppose that $|S_3|=\adim_2(P_r\Box C_s)=3$. By the same reasons (the uniqueness of the eccentric vertex), $S_3$ cannot be a subset of only $\{u_1,u_r\}\times V(C_s)$. Let $(u_i,v_j)\in S_3$ be such that $i\ne 1,r$. Hence, its eccentric vertex, without loss of generality say $(u_r,v_{j+s/2})$, belongs to $S_3$ too. Moreover, also the eccentric vertex of $(u_r,v_{j+s/2})$, which is $(u_1,v_{j})$, is in $S_3$, namely, $S_3=\{(u_i,v_j),(u_r,v_{j+s/2}),(u_1,v_{j})\}$. However, in this case, the vertex $(u_1,v_{j+s/2)}$ forms a class of cardinality $1$ in $\mathcal{Z}_{S_3}$, which is not possible. Consequently, $\adim_2(P_r\Box C_s)\ge 4$. To prove the equality, we consider the set $S_4=\{(u_1,v_0),(u_1,v_1),(u_r,v_{s/2}),(u_r,v_{s/2+1})\}$. Notice that there are four equivalence classes of cardinality $2$ in $\mathcal{Z}_{S_4}$ which are the respective neighbors of $(u_1,v_0)$, $(u_1,v_1)$, $(u_r,v_{s/2})$, $(u_r,v_{s/2+1})$ not in $S_4$. Also, any other equivalence class of $\mathcal{Z}_{S_4}$ has cardinality at least $2$. Therefore, $S_4$ is a $2$-ARS, which concludes the proof of this case. 

Assume next that $r,s$ are not both even. If $s$ is odd, then we readily notice that the vertex $(u_1,v_0)$ forms a $2$-ARS, and so $\adim_2(P_r\Box C_s)=1$. Also, if $s$ is even, then $r$ must be odd. Hence, we again easily see that the vertex $(u_{(r+1)/2},v_0)$ forms a $2$-ARS, and thus $\adim_2(P_r\Box C_s)=1$.

\medskip
\noindent
\textbf{\underline{Case $k=1$:}} If $s$ is even, then the vertex $(u_1,v_0)$ forms a $1$-ARS, and so $\adim_1(P_r\Box C_s)=1$. On the other hand, if $s$ is odd, then any single vertex of $P_r\Box C_s$ is either a $2$-ARS or a $4$-ARS. Thus, $\adim_1(P_r\Box C_s)\ge 2$. To complete the equality, we just observe that the set formed by the two adjacent vertices $(u_1,v_0)$ and $(u_1,v_1)$ is a $1$-ARS, and we deduce that $\adim_1(P_r\Box C_s)=2$.

\subsection*{Proof of Theorem \ref{th:k_antidim_toru}}
Let us see that for every two integers $r,s\ge 3$,
$$\adim_k(C_r\Box C_s)=\left\{\begin{array}{ll}
                                1, & \mbox{if $k=4$ and $r,s$ are odd}, \\
                                2, & \mbox{if $k=4$ and $r,s$ are even}, \\
                                4, & \mbox{if $k=3$ and $r,s$ are even}, \\
                                +\infty, & \mbox{if $k=3$ and $r,s$ are odd}, \\
                                1, & \mbox{if $k=2$ and $r,s$ have distinct parity}, \\
                                4, & \mbox{if $k=2$ and $r,s$ are even},\\
                                \min\{r,s\}, & \mbox{if $k=2$ and $r,s$ are odd},\\
                                1, & \mbox{if $k=1$ and $r,s$ are even}, \\
                                2, & \mbox{if $k=1$ and $r,s$ are not both even}.
                              \end{array}\right.
  $$

\noindent
Similarly to the proof of Theorem \ref{th:k_antidim_cyl}, we again separate our exposition into four cases.

\medskip
\noindent
\textbf{\underline{Case $k=4$:}} Assume $r,s$ are odd. Hence, it can be readily observed that any vertex of $C_r\Box C_s$ is a $4$-ARS, which means $\adim_4(C_r\Box C_s)=1$. On the other hand, if $r,s$ are even, then we consider any vertex $(u_i,v_j)$ and its unique diametral vertex in $C_r\Box C_s$. It can be easily observed that these two vertices form a $4$-ARS. Since any single vertex of $C_r\Box C_s$ forms only a $1$-ARS (when $r,s$ are even), the result $\adim_4(C_r\Box C_s)=2$ follows.

\medskip
\noindent
\textbf{\underline{Case $k=3$:}}  Let $r,s$ be even. Since any single vertex of $C_r\Box C_s$ has a unique diametral vertex, it forms only a $1$-ARS. Thus, if a vertex $(u_i,v_j)$ belongs to some $3$-ARS of $C_r\Box C_s$, then its unique diametral vertex also belong to such set, and so $\adim_3(C_r\Box C_s)\ge 2$. Also, notice that $\adim_3(C_r\Box C_s)=2$ is not possible because any two diametral vertices of $C_r\Box C_s$ form a $4$-ARS of $C_r\Box C_s$ when $r,s$ are even. Thus, $\adim_3(C_r\Box C_s)\ge 3$. In addition, by the same reason, and by the structure of $C_r\Box C_s$, it cannot be $\adim_3(C_r\Box C_s)=3$, and thus, $\adim_k(C_r\Box C_s)\ge 4$. Consider the set  $S=\{(u_0,v_0),(u_0,v_1),(u_{r/2},v_{s/2}),(u_{r/2},v_{s/2+1})\}$. Notice that $(u_0,v_0),(u_{r/2},v_{s/2})$ are diametral as well as $(u_0,v_1),(u_{r/2},v_{s/2+1})$. We shall consider the equivalence classes of $\mathcal{Z}_{S}$. Four of them are as follows, which are those classes in which their vertices are adjacent to vertices of the set $S$.
\begin{align*}
  Z^1 & = \{(u_0,v_{s-1}),(u_1,v_0),(u_{r-1},v_0)\} \\
  Z^2 & = \{(u_0,v_{2}),(u_1,v_1),(u_{r-1},v_1)\} \\
  Z^3 & = \{(u_{r/2},v_{s/2-1}),(u_{r/2+1},v_{s/2}),(u_{r/2-1},v_{s/2})\} \\
  Z^4 & = \{(u_{r/2},v_{s/2+2}),(u_{r/2+1},v_{s/2+1}),(u_{r/2-1},v_{s/2+1})\}
\end{align*}
Notice that these four classes have each cardinality 3. Moreover, by the structure of $C_r\Box C_s$, since $(u_0,v_0),(u_{r/2},v_{s/2})$ are diametral as well as $(u_0,v_1),(u_{r/2},v_{s/2+1})$, it follows that the remaining classes of $\mathcal{Z}_{S}$ (if they exist) have cardinality at least 4. Thus, $S$ is a $3$-ARS of $C_r\Box C_s$, which leads to the desired equality $\adim_3(C_r\Box C_s)=4$.

Suppose next that $r,s$ are odd. Let $S$ be a $3$-ARS of $C_r\Box C_s$ and let $(u_i,v_j)\in S$. Since $r,s$ are odd, the vertex $(u_i,v_j)$ has $4$ diametral vertices, which would either form a whole class of $\mathcal{Z}_{S}$, or part of them will form a whole class of $\mathcal{Z}_{S}$, or every of them are in $S$.
The first situation cannot happen, since there is no vertex of $C_r\Box C_s$, other than $(u_i,v_j)$, having the same distance to all these $4$ such vertices.
In the second situation, it can only happen that exactly three vertices will form a whole class of $\mathcal{Z}_{S}$.
That is, exactly one of such vertices, say $(u_{i'},v_{j'})$, is in $S$.
However, from the other three remaining vertices (being diametral to $(u_i,v_j)$), only two of them have the same distance to $(u_{i'},v_{j'})$.
Thus, $S$ would be only a $2$-ARS, which is not possible.
Consequently, we deduce that the $4$ diametral vertices from $(u_i,v_j)$ must be in $S$.
But then, by the same reasons, the diametral vertices of these previous mentioned $4$ vertices (diametral from $(u_i,v_j)$) must be in $S$ as well.
Therefore, by following this iterative procedure, and due to the structure of the torus $C_r\Box C_s$, in each step at least one new vertex is added to $S$. Thus, it must happen that the whole vertex set of $C_r\Box C_s$ must be in $S$, which is not possible, and it allows to assert that $C_r\Box C_s$ has no $3$-ARS when $r,s$ are odd.

\medskip
\noindent
\textbf{\underline{Case $k=2$:}}  We need to differentiate some situations depending on the parity of $r$ and $s$.
\begin{itemize}
\item \underline{$r,s$ have different parity.} Hence, any vertex $(u_i,v_j)$ of $C_r\Box C_s$ forms a $2$-ARS, since it has exactly two diametral vertices, which will form an equivalence class $Z^i\in \mathcal{Z}_{\{(u_i,v_j)\}}$ of cardinality two, and any other equivalence class with respect to this vertex has cardinality at least $4$. Thus, $\adim_2(C_r\Box C_s)=1$ in such situation.
\item \underline{$r,s$ are even.} Hence, any single vertex forms a $1$-ARS, since it has exactly one diametral vertex.
Thus $\adim_2(C_r\Box C_s)\ge 2$ and let $S'$ be a $2$-ARB of $C_r\Box C_s$. The argument above also means that if a vertex $(u_i,v_j)\in S'$, then also its unique diametral vertex $(u_{i+r/2},v_{j+s/2})$ is in $S'$. The case $\adim_2(C_r\Box C_s)=2$ is then not possible, because any two diametral vertices of $C_r\Box C_s$ form a $4$-ARS. Thus, $\adim_2(C_r\Box C_s)\ge 3$. But then, again the same situation about the diametral vertices (and the structure of $C_r\Box C_s$) implies that $\adim_2(C_r\Box C_s)\ge 4$. We now consider the set $S''=\{(u_0,v_0),(u_0,v_{s/2}),(u_{r/2},v_0),(u_{r/2},v_{s/2})\}$. Notice that $(u_0,v_0),(u_{r/2},v_{s/2})$ are diametral in $C_r\Box C_s$, as well as $(u_0,v_{s/2}),(u_{r/2},v_0)$. Moreover, the vertices of the copies of $C_s$ corresponding to the vertices $v_0$ and $v_{r/2}$ in $C_r\Box C_s$ form $s-2$ equivalence classes of cardinality $2$ in $\mathcal{Z}_{S''}$, as well as, the vertices of the copies of $C_r$ corresponding to the vertices $u_0$ and $u_{s/2}$ in $C_r\Box C_s$ form $r-2$ equivalence classes of cardinality $2$. The remaining classes of $\mathcal{Z}_{S''}$ have cardinality $4$. As a consequence, we deduce that $S''$ is a $2$-ARS. Therefore $\adim_2(C_r\Box C_s)=4$.
\item \underline{$r,s$ are odd.} This part of the proof uses a somehow different technique. To this end, we shall first present some necessary tools and terminologies in order to simplify the notation. By a \emph{row} or a \emph{column} of $T{=C}_r \square {C}_s$ we mean a copy of $C_r$ or of $C_s$ in $T$.

\begin{property}
\label{farthest1}
For any vertex $x=(u_i, v_j)$ the set of eccentric vertices of $x$ is
\begin{equation*}
    \begin{split}
    \{x^{\text{\tiny{-\,-}}}, x^{\text{\tiny{+\,-}}}, x^{\text{\tiny{+\,+}}}, x^{\text{\tiny{-\,+}}} \},\label{proper1}
    \end{split}
\end{equation*}
\end{property}
where $x^{\text{\tiny{-\,-}}}=(u_{i-\frac{r-1}{2}}, v_{j-\frac{s-1}{2}})$; $ x^{\text{\tiny{+\,-}}}= (u_{i+\frac{r-1}{2}}, v_{j-\frac{s-1}{2}})$; $x^{\text{\tiny{+\,+}}}= (u_{i+\frac{r-1}{2}}, v_{j+\frac{s-1}{2}})$; and, $x^{\text{\tiny{-\,+}}}=(u_{i-\frac{r-1}{2}}, v_{j+\frac{s-1}{2}})$.\\

From Property \ref{farthest1} we immediately obtain that the $\adim_2(C_r\Box C_s)\ge 2$.

First note that, any row of $T$, $U_i=\{u_i\} \square V({C}_s)$, $1\leq i\leq r$ is a $2$-ARS for $T$, since for each $j\in\{1, \dots, s\}$ the pair of vertices $\{(u_{i+\delta}, v_j), (u_{i-\delta}, v_j)\}$ with $1 \le \delta \le \frac{r-1}{2}$
forms an equivalence class in ${\mathcal{Z}}_{U_i}$. Similarly, any column of $T$, $V_j=\{v_j\} \square V({C}_r)$, $1\leq j\leq s$ is a $2$-ARS for $T$ since for each $i\in\{1, \dots, r\}$ the pair of vertices $\{(u_{i}, v_{j-\delta}), (u_{i}, v_{j+\delta})\}$ with $1 \le \delta \le \frac{s-1}{2}$ forms an equivalence class in ${\mathcal{Z}}_{V_j}$. Thus, $2\leq \adim_2(C_r\Box C_s)\le \min\{r,s\}$. \\
Hence, we next focus on sets of vertices $S$ such that $2\leq |S| < \min \{r, s\}$. In the following, we will say that the vertices of a given set are \emph{aligned} if all of them belong to the same row or to the same column of $T$, even if they are not necessarily consecutive. 

The next result gives conditions on aligned vertices, which determine when one of them may define a $1$-ARS. 

\begin{property}
\label{pro_alligned}
Any vertex set $S\subset V(T)$  that is neither a full row nor a full column of $T$, containing 
aligned vertices is a $1$-ARS for $T$.
$($see the proof at the end of the proof of Theorem \ref{th:k_antidim_toru}$)$.
\end{property}

We also have the following property when $S$ does not contain aligned vertices.
\begin{property} \label{pro_no_alligned}
Any vertex set $S$, $2 \leq |S| < \min \{r, s\}$, containing no aligned vertices is a $1$-ARS.
$($see the proof at the end of the proof of Theorem \ref{th:k_antidim_toru}$)$.
\end{property}

Properties \ref{pro_alligned} and \ref{pro_no_alligned} above allow to deduce that $S$ is indeed a $1$-ARS for $T$.
Therefore, we finally conclude that $\adim_2(C_r\Box C_s)= \min\{r,s\}$ when $r,s$ are odd.
\end{itemize}

\medskip
\noindent
\textbf{\underline{Case $k=1$:}} If $r,s$ are even, then any vertex forms a $1$-ARS, since it has exactly one diametral vertex. Thus, $\adim_1(C_r\Box C_s)=1$ in this case. Now, if $r,s$ are not both even, then consider for instance $s$ is odd. Hence, the set $S_1=\{(u_0,v_0),(u_0,v_1)\}$ (notice that $(u_0,v_0),(u_0,v_1)$ are diametral in $C_s$), is a $1$-ARS of $C_r\Box C_s$, since the vertex $(u_0,v_{(s+1)/2})$ forms an equivalence class $Z^i\in \mathcal{Z}_{S_1}$. Thus, $\adim_1(C_r\Box C_s)\le 2$. Also, as any single vertex of $C_r\Box C_s$ is either a $2$-ARS (when $r,s$ are of distinct parity), or a $4$-ARS (when $r,s$ are odd), we deduce that $\adim_1(C_r\Box C_s)=2$, which completes the proof.
\qed

\subsection*{Proof of Property \ref{pro_alligned}}

\noindent
Let us see that any vertex set $S\subset V(T)$  that is neither a full row nor a full column of $T$, containing aligned vertices is a $1$-ARS for $T$.
In the proof we use the notation $p_{xy}$ to denote the shortest path from $x$ to $y$ in $T$. We will also use the fact for three given aligned vertices $x^-, x^+,z\in U_i$, $1\leq i\leq r$. If $x^-$ does not belong to $p_{x^+z}$ and $x^+$ does not belong to $p_{x^-z}$, then the class of $z$ relative to $S_0=\{x^-, x^+\}$ is a singleton. Therefore, the class of $z$ relative to any set $S\supseteq S_0$ such that $z\notin S$ is a singleton as well.\\
Without loss of generality, we assume that $S$ contains two aligned vertices in the same row $U_i$. Otherwise, the proof should just interchange rows and columns.\\
 Since $S$ is not a full row, then $\overline{U}_{i} = U_{i} \setminus S \neq \emptyset$, and so there exists $z = (u_i, v_j)\in \overline{U}_i$. \\
 Let $P\equiv x^- \dots z \dots x^+$ be the only path containing $z$ such that its end-vertices $x^-, x^+\in S$ and all its intermediate vertices belong to $\overline U_i$.  Consider the possible cases:
 \begin{itemize}
      \item $P$ is the shortest path from $x^-$ to $x^+$.
      Thus, $z\in P=p_{x^-, x^+}$.
      Then $x^-\notin p_{x^+,z}$ and $x^+\notin p_{x^-,z}$. By the remark above, the class $Z^{z}$ is the singleton $\{z\}$.
      \item $P$ is not the shortest path from $x^-$ to $x^+$.
     Let $E^-=\arg\max\{ d(x^-, y): y\in U_i\}=\{y^-, y^+\}$, where $y^-$ is such that $x^+\notin p_{x^-y^-}$ and $x^+\in p_{x^-y^+}$. Since $|P|\geq\frac{r-1}{2}$, it follows $y^-,y^+\in P\setminus S$.
     Indeed, $x^-\notin p_{x^+,y^-}$. Otherwise, the length of the path from $x^+$ to $y^-$ containing $x^-$ would be greater than $ d(x^-, y^-)\geq \frac{r-1}{2}$, contradicting that $p_{x^+,y^-}$ is a shortest path in $U_i$.
   \end{itemize}
The arguments of the two cases above allow to conclude that $S$ is indeed a $1$-ARS for $T$.
\qed

\subsection*{Proof of Property \ref{pro_no_alligned}}

\noindent
Let us see that any vertex set $S\subset V(T)$  that is neither a full row nor a full column of $T$, containing aligned vertices is a $1$-ARS for $T$.\\
Let $x=(u_{i}, v_j)\in S$. Due to the symmetry of the torus, we can assume without loss of generality that $x = (u_\frac{r-1}{2}, v_\frac{s-1}{2})$. Consider the following cases:
\begin{itemize}
    \item  \underline{$(u_{i+1}, v_{j+1})\in S$.}\\
    Let $y = (u_{i+1}, v_{j+1})=(u_{\frac{r-1}{2}+1}, v_{\frac{s-1}{2}+1})$ and $x^* = (u_{0}, v_{\frac{s-1}{2}})$. Since $\min \{r, s\}\geq 3$, then $x^*\ne x$ and $x^*\ne y$. Moreover, $x^*\notin S$ because $x$ and $x^*$ are aligned.
    A shortest path from $x^*$ to $x$ is $(u_0, v_j)-(u_{1}, v_j)\dots (u_{j}, v_{j})$ so $d_{x^*, x}=\frac{r-1}{2}+1$. Moreover, a shortest path from $x^*$ to $y$ is $(u_0, v_j)-(u_{s-1}, v_j)-(u_{s-1}, v_{j+1})\dots (u_{j+1}, v_{j+1})$ so  $d_{x^*, y}=1+1+\left((r-1)-(\frac{r-1}{2}+1)\right)$. Finally, it can be observed that there is no other vertex $z\ne x^*$ such that $d_{z, x}=\frac{r-1}{2}+1$ and $d_{z, y}=2+\left((r-1)-(\frac{r-1}{2}+1)\right)$.
    Hence, the singleton $\{(u_{0}, v_{\frac{s-1}{2}})\}$ forms an equivalence class of ${\mathcal{Z}}_S$.
  \item  \underline{$(u_{i+1}, v_{j+1})\notin S$.} We need now two different situations.
    \begin{itemize}\item[(i)] \underline{$(u_0, v_0)\in S$.}\\
     Let now $y= (u_0, v_0)$ and $x^* = (u_{\frac{r-1}{2}+1}, v_{\frac{s-1}{2}+1})$, which by hypotheses does not belong to $S$.\\
   Furthermore, $x^*\ne y$ because $\min\{r, s\} \ge 3$.  Indeed, $d_{x, x^*} = 2$, and $d_{y, x^*} = \frac{r+s}{2}$, since no shortest path from $y$ to $x^*$ contains $x$, as all such paths \emph{leave} $y$ in a direction \emph{opposite} to $x$. Moreover, it can now be observed that $x^*$ is the only vertex such that $d_{x, x^*} = 2$ and $d_{y, x^*} = \frac{r+s}{2}$.
    Hence the singleton $\{x^*\}$ forms an equivalence class of ${\mathcal{Z}}_S$.
  \item[(ii)] \underline{$(u_0, v_0)\notin S$.}\\
      By the symmetry of the torus we can now assume that $y = (u_i, v_j) \in S$ with $i < \frac{r-1}{2}$ and $j < \frac{s-1}{2}$.\\
      Let now $x^*=x^{\text{\tiny{-\,-}}}$ and note that the only vertices at the same distance from $x$ as $x^{{\text{\tiny{-\,-}}}}$ are $x^{{\text{\tiny{+\,-}}}}$, $x^{\text{\tiny{+\,+}}}$, and $x^{\text{\tiny{-\,+}}}$. However,
      $$d_{y, x^{\text{\tiny{-\,-}}}}  = d_{y, x^{\text{\tiny{+\,-}}}} -1 = d_{y, x^{\text{\tiny{-\,+}}}}-1 = d_{y, x^{\text{\tiny{+\,+}}}}-2.$$
      Hence, there is no other vertex $z$ such that $d_{x, x^*} = d_{z, x^*}$ and $d_{y, x^*} = d_{y, x^*}$, and so, we conclude that the singleton $\{x^*\}$ forms an equivalence class of ${\mathcal{Z}}_S$.
      \end{itemize}
\end{itemize}
\qed

\subsection*{Proof of Theorem \ref{th:k_antidim_Hamming}}

Let us see that for every $r\ge 4$,
$$\adim_k(K_r\Box K_r)=\left\{\begin{array}{ll}
                                3, & \mbox{if $k=1$}, \\
                                2, & \mbox{if $k=2$}, \\
                                r-k, & \mbox{if $3\le k\le r-2$}, \\
                                r, & \mbox{if $k=r-1$}, \\
                                +\infty, & \mbox{if $r\le k\le 2r-3$}, \\
                                1, & \mbox{if $k=2r-2$}.
                              \end{array}\right.
  $$

\medskip
\noindent
\textbf{\underline{Case $k=1$:}} Since every vertex of $K_r\Box K_r$ forms a $(2r-2)$-ARS and $2r-2\ge 6$, we deduce that $\adim_1(K_r\Box K_r)\ge 2$.  Let $S=\{(u_i,v_j),(u_k,v_l)\}$ be a $1$-ARB and consider the equivalence classes $\mathcal{Z}_S=\{Z^1,\dots,Z^t\}$, for some $t\ge 1$, defined by the equivalence relation $\mathcal{R}_S$. If either $i=k$ or $j=l$, say w.l.g. $i=k$ (notice that it cannot happen that both $i=k$ and $j=l$), then every equivalence class $Z^q\in \mathcal{Z}_S$ has at least $r-2$ vertices. Indeed, such equivalence class consists of the $r-2$ vertices of the form $(u_i,v_\alpha)$. Thus, $S$ is an $(r-2)$-ARS. Since $r\ge 4$, $S$ is a $k$-metric ARS for some $k\ge 2$, which is not possible. If $i\ne k$ and $j\ne l$, then the two vertices $(u_i,v_l)$ and $(u_k,v_j)$ are in the same equivalence class defined by $\mathcal{Z}_S$. Moreover, any other equivalence class in $\mathcal{Z}_S$ has more than two vertices. Thus, $S$ is a $2$-ARS, which again is not possible. As a consequence of these two contradictions, it must hold that $\adim_1(K_r\Box K_r)\ge 3$. On the other hand, the set of three vertices $S_1=\{(u_1,v_1), (u_1,v_2), (u_2,v_1)\}$ forms a $1$-ARS since the vertex $(u_2,v_2)$ forms an equivalence class of $\mathcal{Z}_{S_1}$. Thus, $\adim_1(K_r\Box K_r)\le 3$.

\medskip
\noindent
\textbf{\underline{Case $k=2$:}}  We observe that the set $S_2=\{(u_1,v_1), (u_2,v_2)\}$ forms a $2$-ARS, since the vertices $(u_1,v_2)$, $(u_2,v_1)$ form an equivalence class of $\mathcal{Z}_{S_2}$.
Thus, $\adim_2(K_r\Box K_r)\le 2$. Also $\adim_2(K_r\Box K_r)\ge 2$ holds, and thus the equality follows in this case.

\medskip
\noindent
\textbf{\underline{Case $3\le k\le 2r-3$:}}
Let $X$ be a $k$-ARB of $K_r\Box K_r$. Note that $|X|\ge 2$. Suppose $X$ contains two vertices $(u_i,v_j)$ and $(u_k,v_l)$ such that $i\ne k$ and $j\ne l$. We consider the equivalence classes of $\mathcal{Z}_{X}$. If there is a class in $\mathcal{Z}_{X}$ with all vertices at distance $1$ to $(u_i,v_l)$ and $(u_k,v_j)$, then we obtain a contradiction, because there are only two vertices with such property, which are $(u_i,v_l)$ and $(u_k,v_j)$. This would mean that $X$ is either a $2$-ARS (if neither $(u_i,v_l)$ nor $(u_k,v_j)$ are in $X$), or a $1$-ARS (if either $(u_i,v_l)\in X$ or $(u_k,v_j)\in X$). As a consequence, the two vertices $(u_i,v_l)$ and $(u_k,v_j)$ should be in $X$ as well. By similar arguments, there does not exist a third vertex $(u_{i'},v_{j'})\in X$ such that $i\ne i'\ne k$ and $j\ne j'\ne l$. Thus, $X\subset (V(K_r)\times \{v_j,v_l\})\cup (\{u_i,u_k\}\times V(K_r))$. If there are two vertices $(u_{a},v_b),(u_c,v_d)\in X$ such that $a\ne c$, $b\ne d$, and w.l.g. $a\ne i,k$ (which means $b\in \{j,l\}$) and $d\ne j,l$ (which means $c\in \{i,k\}$), then, using similar arguments as above, we deduce that also $(u_{a},v_d),(u_c,v_b)\in X$. However, this is a  contradiction since $(u_{a},v_d)\notin (V(K_r)\times \{v_j,v_l\})\cup (\{u_i,u_k\}\times V(K_r))$. This means that either $X\subset (V(K_r)\times \{v_j,v_l\})$ or $X\subset (\{u_i,u_k\}\times V(K_r))$. Assume $X\subset (V(K_r)\times \{v_j,v_l\})$ (the other case is symmetric). Using again the same arguments as those used to prove that the vertices $(u_i,v_l)$ and $(u_k,v_j)$ are in $X$, we can show that $(u_h,v_j)\in X$ for some $h\in \{1,\dots,r\}$ if and only if $(u_h,v_l)\in X$. Let $X_j=X\cap (V(K_r)\times \{v_j\})$ and $X_l=X\cap (V(K_r)\times \{v_l\})$ and note that $|X_j|=|X_l|=t$ for some $2\le t\le r$. Hence, the equivalence classes of $\mathcal{Z}_{X}$ are as follows:
\begin{itemize}
  \item Two equivalence classes of cardinality $r-t$, which are the sets $(V(K_r)\times \{v_l\})\setminus (X_l)$ and $(V(K_r)\times \{v_j\})\setminus (X_j)$. These classes could be empty when $t=r$.
  \item $t$ classes of cardinality, $r-2$ which are the sets $\{u_\alpha\}\times V(K_r)\setminus X$ for every $1\le \alpha\le r$ such $(u_\alpha,v_j)\in X$.
  \item A class with all the remaining vertices of cardinality $(r-t)(r-2)$, which could also be  empty if $t=r$.
\end{itemize}
As a consequence, we obtain that $X$ is an $(r-t)$-ARS (indeed an $(r-t)$-ARB), when $t\le r-1$ or an $(r-2)$-ARB when $t=r$. Therefore, we conclude that $k=r-t$ or $k=r-2$ (accordingly). However, if we consider any of the two subsets $X_j$ or $X_l$ of $X$, we observe that both of them are $(r-t)$-ARSs (or $(r-2)$-ARSs), which means that $X$ is not a $k$-ARB, and this is a final contradiction with our first assumption, namely, $X$ does not contain two vertices $(u_i,v_j)$ and $(u_k,v_l)$ such that $i\ne k$ and $j\ne l$.

Consequently, either $X\subset (V(K_r)\times \{v_j\})$ for some $v_j\in V(K_r)$, or $X\subset (\{u_i\}\times V(K_r))$ for some $u_i\in V(K_r)$, say w.l.g. that $X\subset (V(K_r)\times \{v_j\})$. In such a situation, we see that the vertices $(V(K_r)\times \{v_j\})\setminus X$ form an equivalence class of $\mathcal{Z}_{X}$ of cardinality $r-|X|$ when $|X|\le r-1$, and every other equivalence class of $\mathcal{Z}_{X}$ has more vertices. Since $X$ is a $k$-ARB, we deduce that $k=r-|X|$, which means $\adim_{k}(K_r\Box K_r)=|X|=r-k$ for every $3\le k\le r-2$. For the case $|X|\le r-1$ we also have that $k\le r-2$ since $|X|\ge 2$.
If $|X|=r$, then $X$ is an $(r-1)$-ARB, which means $k=r-1$, and so $\adim_{r-1}(K_r\Box K_r)=|X|=r$. In addition, these arguments show that there does not exist $k$-ARSs in $K_r\Box K_r$ for any $r\le k\le 2r-3$.

Finally, any vertex of $K_r\Box K_r$ forms a $(2r-2)$-ARS, which means $\adim_{2r-2}(K_r\Box K_r)=1$. This concludes the proof.
\qed

\end{document}